\documentclass{amsart}
\usepackage{ amssymb }
\usepackage{xcolor}
\usepackage{graphicx}
\usepackage[T1]{fontenc}
\usepackage{ stmaryrd }
\usepackage{enumitem}
\usepackage{rotating}
\usepackage{tikz}
\usepackage[normalem]{ulem}
\usetikzlibrary{decorations.pathreplacing}
\usetikzlibrary{snakes}
\usepackage{pict2e}

\usepackage{hyperref}
\hypersetup{colorlinks=true,linkcolor=blue,citecolor=blue}

\newcommand{\IP}{{\mathbb P}}
\newcommand{\IE}{{\mathbb E}}
\newcommand{\DP}{{\mathrm P}}
\newcommand{\DE}{{\mathrm E}}

\renewcommand{\b}{\beta}

\newcommand{\cvlaw}{\stackrel{{ (d)}}{\longrightarrow}}
\newcommand{\eqlaw}{\stackrel{(d)}{=}}

\newcommand{\dd}{\mathrm{d}}

\newcommand{\balpha}{{\bar \alpha}}
\newcommand{\bbeta}{{\bar \beta}}
\newcommand{\bgamma}{{\bar \gamma}}
\newcommand{\bdelta}{{\bar \delta}}
\newtheorem{theorem}{Theorem}[section]
\newtheorem{lemma}[theorem]{Lemma}
\newtheorem{proposition}[theorem]{Proposition}

\newtheorem{corollary}[theorem]{Corollary}

\newtheorem*{theorem*}{Theorem}
\newtheorem*{lemme*}{Lemme}

\newtheorem*{proposition*}{Proposition}
\newtheorem{conjecture}[theorem]{Conjecture}

\newtheorem{remark}[theorem]{Remark}

\theoremstyle{definition}

\theoremstyle{remark}

\keywords{Two dimensional subcritical directed polymer, high moments of partition functions,  planar random walk intersections, maxima of log-correlated fields}
\subjclass[2010]{Primary 82B44 secondary  82D60, 60G50, 	60H15}

\begin{document}
\title[Moments of polymers]{Moments of partition functions of 2D Gaussian polymers in the weak disorder regime - I}
\author{Cl\'ement Cosco and Ofer Zeitouni}
\address{Cl\'{e}ment Cosco,
Ceremade, Universite Paris Dauphine, Place du Mar\'{e}chal de Lattre de Tassigny, 75775 Paris Cedex 16, France}
\address{Ofer Zeitouni,
Department of Mathematics, Weizmann Institute of Sciences,
Rehovot 76100, Israel.}
\thanks{This project has received funding from the European Research Council (ERC) under the European Union's Horizon 2020 research and innovation programme (grant agreement No. 692452). The first version of this work was completed while the first author was with the Weizmann Institute.}
\email{clement.cosco@gmail.com, ofer.zeitouni@weizmann.ac.il}

\begin{abstract}
Let $W_N(\b) = \DE_0\left[e^{ \sum_{n=1}^N \b\omega(n,S_n) - N\b^2/2}\right]$ be the partition function of a two-dimensional directed polymer in a random environment,
where $\omega(i,x), i\in \mathbb{N}, x\in \mathbb{Z}^2$ are i.i.d.\ standard normal and $\{S_n\}$ is the path of a random walk. With $\beta=\beta_N=\hat\b \sqrt{\pi/\log N}$ and $\hat \b\in (0,1)$ (the subcritical window),
$\log W_N(\b_N)$ is known to converge in distribution to a Gaussian  law of mean $-\lambda^2/2$ and variance $\lambda^2$, with $\lambda^2=\log (1/(1-\hat\b^2)$
 (\textit{Caravenna,  Sun, Zygouras, Ann.\ Appl.\ Probab.\ (2017)}).
We study in this paper
the moments $\IE [W_N( \b_N)^q]$ in the subcritical window, for $q=O(\sqrt{\log N})$. The analysis is based on ruling out triple intersections
\end{abstract}
\maketitle
\section{Introduction and statement of results}
We consider in this paper the partition function of two dimensional
directed polymers in
Gaussian environment, and begin by introducing the model.
Set
\begin{equation}
  \label{eq-WN}
  W_N(\b,x) = \DE_x\left[e^{ \sum_{n=1}^N \b\omega(n,S_n) - N\b^2/2}
  \right], \quad  x\in \mathbb Z^d.
\end{equation}
Here, $\{\omega_{n,x}\}_{n\in \mathbb{Z}_+, x\in \mathbb{Z}^d}$ are
i.i.d. standard centered
Gaussian random variables of law $\IP$, $\{S_n\}_{n\in \mathbb{Z}_+}$
is simple random walk, and $\DE_x$ denotes the law of simple random walk started
at $x\in \mathbb{Z}^2$. Thus, $W_N(\b,x)$ is a random variable measurable
on the $\sigma$-algebra
$\mathcal{G}_N:=\sigma\{ \omega(i,x): i\leq N, x\in \mathbb{Z}^d\}.$
For background, motivation and results on the rich theory
surrounding this topic, we refer the reader to \cite{CStFlour}.
In particular, we mention the relation with the $d$ dimensional stochastic
heat equation (SHE).

The random variables $W_N(\b,0)$ form a $\mathcal{G}_N$ positive martingale,
and therefore converge almost surely to a limit $W_\infty(\b,0)$.
It is well known that in dimensions $d=1,2$, for any $\b>0$ we have $W_\infty(\b,0)=0$, a.s., while for $d\geq 3$, there exists $\beta_c>0$ so that
$W_\infty(\b,0)>0$ a.s. for $\b<\b_c$ and $W_\infty(\b,0)=0$ for
$\b> \b_c$. We refer to these as the \textit{weak} and \textit{strong}
disorder regimes, respectively. In particular, for $d=2$, which is our focus in
this paper, for any $\b>0$, we are in the strong disorder regime.

A meaningful rescaling in dimension $2$ was discovered in the context of the SHE by Bertini and Cancrini \cite{Bertini98} and was later generalized by Caravenna, Sun and Zygouras \cite{CaraSuZy-universalityrelev}, in both the SHE and polymer setups, to a wider range of parameters for which a phase transition occurs. See also \cite{CaSuZy18,CaSuZyCrit21,ChDu18,Gu18KPZ2D,NaNa21}.
Introduce the mean intersection local time for random walk
\begin{equation}
  \label{eq-RNas}
  R_N = \DE_0^{\otimes 2}\left[\sum_{n=1}^N \mathbf{1}_{S_n^1 = S_n^2}\right]\sim \frac{\log N}{\pi}.\end{equation}
The asymptotic behavior of $R_N$ follows from
the local limit theorem \cite[Sec. 1.2]{LawlerIntersections}. Further, the Erd\H os-Taylor theorem \cite{ErdosTaylor} states that $\frac{\pi }{\log N}\sum_{n=1}^N \mathbf{1}_{S_n^1=S_n^2}$ converges in distribution to an exponential random variable of parameter 1.

Set
\begin{equation} \label{eq:defBeta_N}
\b_N =  \frac{\hat{\b}}{\sqrt {R_N}}, \quad \hat \b\geq 0.
\end{equation}
We will use the short-notation $W_N = W_N(\b_N,0)$.
With it, see \cite{CaraSuZy-universalityrelev}, one has
\begin{equation} \label{eq:pointwiseGaussianCV}
  \forall \hat{\b}<1:\quad \log W_N \cvlaw \mathcal N\left(-\frac {\lambda^2} 2,\lambda^2\right), \quad \mbox{\rm with} \quad \lambda^2= \lambda^2(\hat \b) = \log \frac 1 {1-{\hat \b}^2}\ .
\end{equation}
The convergence in \eqref{eq:pointwiseGaussianCV} has recently been
extended in \cite{LyZy21} to the convergence of $W_N$ to the exponential of a Gaussian, in all $L^p$.
(The critical case $\hat\b=1$, which we will not study in this paper,  has received considerable attention, see \cite{Bertini98,CaSuZyCrit18,CaSuZyCrit21,GQT}.)

The spatial behavior of $W_N(\b_N,x)$ is also of interest.
Indeed, one has, see \cite{CaSuZy18},
\begin{equation} \label{eq:GFFlimit}
G_N(x):=\sqrt{R_N} \left(\log W_N(\b_N,x\sqrt N) - \IE \log W_N(\b_N,x\sqrt N)\right) \cvlaw \sqrt{\frac {\hat {\b}^2} {1-{\hat \b}^2}} G(x),
\end{equation}
with $G(x)$ a log-correlated Gaussian field on $\mathbb R^2$. The convergence
in \eqref{eq:GFFlimit} is in the weak sense, i.e.\
for any smooth, compactly supported function  $\phi$,
$\int \phi(x) G_N(x) dx$ converges to a centered
Gaussian random
variable of variance $\hat{\b}^2 \sigma^2_\phi/(1-\hat{\b}^2)$, where
\begin{equation}
  \label{eq:lim-cov}
  \sigma^2_\phi=\frac{1}{2\pi}\iint \phi(x)\phi(y)\int_{|x-y|^2/2}^\infty
  z^{-1}e^{-z} dz.\end{equation}
  One recognizes $\sigma^2_\phi$ in \eqref{eq:lim-cov}
  as the variance of the integral of
  $\phi$ against the solution of the \textit{Edwards-Wilkinson} equation. For a related result in the KPZ/SHE setup, see \cite{CaSuZy18,Gu18KPZ2D,NaNa21}.

  Logarithmically correlated fields, and
  in particular their extremes and large values,
  have played an important recent
  role in the study of various models of probability theory at the critical
  dimension,
  ranging from their own study \cite{Biskup,BDZ,DRSV,RV},
  random walk and Brownian motion
  \cite{BRZ,DPRZ}, random matrices \cite{CMN,CN,CFLW},
  Liouville quantum
  gravity \cite{DuSh,KRV}, turbulence \cite{GRV}, and more.
  In particular, exponentiating Gaussian
  logarithmically
  correlated fields yields Gaussian multiplicative chaoses,
  with the ensuing question of convergence towards them.

 In the context of polymers,
\eqref{eq:GFFlimit} opens the door to the study of such questions.
A natural role is played by the random measure
\[\mu_{N}^{\gamma}(x)=\frac{e^{\gamma G_N(x)}}
{\IE e^{\gamma G_N(x)}},\]
and it is natural to ask about its convergence towards a GMC, and
about extremes of $G_N(x)$ for $x$ in some compact subsets of $\mathbb R^2$.

A preliminary step toward any such analysis involves evaluating exponential
moments of $G_N(0)$. This is our goal in this paper. In the following, $q=q(N)$ denotes an integer $q\geq 2$ that may depend on $N$.
Our main result is the following.
\begin{theorem} \label{th:main}
There exists $\hat \beta_0\leq 1$ so that if $\hat \beta<\hat\beta_0$ and
  \begin{equation} \label{eq:condition20nr}
\limsup_{N\to\infty} \frac{3\hat \beta^2}{\left(1-\hat \beta^2\right)}  \frac{1}{\log N}    \binom{q}{2} < 1,
\end{equation}
then,
\begin{equation} \label{eq:estimate}
\IE[W_N^q]  \leq e^{{\binom{q}{2}\lambda^2
(1+\varepsilon_N)}},
\end{equation}
where $\varepsilon_N=\varepsilon(N,\hat \beta) \searrow 0$ as $N\to\infty$.
\end{theorem}
The proof will show that in Theorem \ref{th:main}, $\hat \beta_0$ can be taken as $1/96$, but we do not expect this to be optimal.
\begin{remark}
  \label{rem-1.2}
  With a similar method, we can also prove that the estimate
\eqref{eq:estimate}
holds for all $\hat \beta <1$ at the cost of choosing
$q^2=o(\log N / \log \log N)$,
see Section \ref{subsec-rem1.2} for details.
In particular, we obtain that the partition function possesses
all (fixed) moments  in the region $\hat \beta <1$:
\begin{equation} \label{eq:Lp-regions}
\forall q\in \mathbb N,\quad  \sup_{N} \IE[W_N^q] < \infty.
 \end{equation}
As mentioned above,
\eqref{eq:Lp-regions} was
independently proved  in \cite{LyZy21}.
(See also \cite{LyZy22} for further precision and a multivariate generalization of
the Erd\H{o}s-Taylor theorem.)
They also observed that together with
the convergence in distribution \eqref{eq:pointwiseGaussianCV},
the estimate \eqref{eq:Lp-regions}
implies that for all fixed $q\in \mathbb N$,
\begin{equation}
  \label{eq-convergencepol}
  \IE[W_N^q] e^{-\binom{q}{2} \lambda^2(\hat \beta)} \underset{N\to\infty}{\longrightarrow} 1.
\end{equation}
Note however that the estimate \eqref{eq:estimate} does \textit{not}
yield \eqref{eq-convergencepol} when $q\to \infty$ with $N\to\infty$.
\end{remark}
\begin{remark}
  \label{rem-1.3}
Theorem \ref{th:main} is of course not enough to prove convergence toward
a GMC.
  For that, one would need to improve the error 
  in the exponent from $O(q^2 \varepsilon_N)$ to $O(1)$, to obtain a complementary lower bound  and, more important, to derive
  similar multi-point
  estimates. We hope to return to these issues in future work.
\end{remark}
The structure of the paper is as follows. In the next
Section \ref{sec-Intersection},
we use a well-worn argument to reduce the computation of moments
to certain estimates concerning the intersection of (many) random walks.
After some standard preliminaries, we state there our main
technical estimate, Theorem \ref{th:momentwithoutTriple}, which provides
intersection estimates under the extra assumptions that
all intersections are in pairs, i.e.\ that no triple (or more) points exist. The rest
of the section provides the proof of Theorem \ref{th:main}.
Section \ref{sec-notriple} then develops the induction scheme that is used
to prove  Theorem \ref{th:momentwithoutTriple}. Since
we assume that there are no triple (or more) intersections, we may consider
particles as matched in pairs at intersection times. The induction is
essentially on the number of instances in which ``matched particles''
break the match and create a different matching.
Section \ref{sec-discussion} provides a discussion of our results, their
limitations, and possible extensions. In particular we explain there
why the constraint on $q$ in
Theorem \ref{th:main} limits
our ability to obtain the expected sharp upper bounds on the maximum
of $\log W_N(\hat \beta_N, x\sqrt{N})$.
The appendices collect several auxilliary results and a sharpening
of one of our estimates,
see Proposition \ref{prop:factor2ndMomentbis}.

\noindent
{\bf Acknowledgment} We thank Dimitris Lygkonis and Nikos Zygouras for sharing their work \cite{LyZy21} with us prior to posting,  and for useful comments. 
We thank the referee for a
careful reading of the original manuscript and for
many comments that helped us greatly improve the paper. We are grateful to Shuta Nakajima for helpful comments on a previous version of the article.

\noindent
{\bf Data availability statement} No new data was generated in relation to this manuscript.

\noindent
{\bf Conflict of interest statement} There are no conflict of interests to either author.

\section{Intersection representation, reduced moments, and proof of Theorem
\ref{th:main}}
\label{sec-Intersection}
Throughout the rest of the paper, we let $p(n,x)=p_n(x)=\DP_0(S_n=x)$.
There is a nice formula for the $q$-th moment of the partition function
whose importance is apparent in previous work on directed polymers, for
example in \cite{CaravennaFrancesco2019TDsr,CaSuZyCrit18}.
Indeed,
\begin{align*}
\IE[W_N^q] & = \DE_0^{\otimes q}
\IE e^{\sum_{i=1}^q \sum_{n=1}^N (\beta_N \omega(n,S_n^{i}) - \b_N^2/2)},
\end{align*}
where $S^1,\dots,S^q$ are $q$ independent copies of the simple random walk and $\DE_{X}^{\otimes q}$ denotes the expectation for the product measure started at $X = (x^1,\dots,x^q)$. (If the starting point $X$ is not specified, we assume $X=\mathbf 0$.)
Since the $\omega(i,x)$ are Gaussian and the variance of $\sum_{i=1}^q \beta_N \omega(n,S_n^{i})$ is equal to $\beta_N^2 \sum_{i=1\dots q,j=1\dots q} \mathbf{1}_{S_n^i = S_n^j}$, we have the following  formula for the moment in terms of intersections of $q$ independent random walks:
\begin{equation} \label{eq:momentFormula}
\IE[W_N^q] = \DE^{\otimes q}\left[ e^{\beta_N^2 \sum_{1\leq i<j\leq q} \sum_{n=1}^N \mathbf{1}_{S_n^{i} = S_n^j}}\right].
\end{equation}

\subsection{No triple estimate}
The key step in upper bounding the right-hand side of
\eqref{eq:momentFormula}  is to restrict the summation
to subsets where
there are no triple (or more) intersections. More precisely, denote by
\begin{align}
 F_n& = \left\{\exists (\balpha,\bbeta,\bgamma)\in \llbracket 1,q\rrbracket^3:  \balpha< \bbeta<\bgamma,  S_n^\balpha = S_n^\bbeta = S_n^\bgamma\right\},  \label{eq-Fn} \\
K_n&=\big\{\exists (\balpha,\bbeta,\bgamma,\bdelta)\in \llbracket 1,q\rrbracket^4:  \balpha< \bbeta, \bgamma<\bdelta,\label{eq-Kn}\\
& \hspace{2cm} \{\balpha,\bbeta\}\cap \{\bgamma,\bdelta\}=\emptyset, S_n^\balpha = S_n^\bbeta,   S_n^\bgamma=S_n^\bdelta\big\}, \nonumber
\end{align}
and let
\[G_T = \bigcap_{n\in \llbracket 1,T\rrbracket} (F_n\cup K_n)^\complement\]
be the event that there is no triple (or more) intersection, i.e. that at 
each given time no more than a pair of particles are involved in an intersection.

The following theorem is the technically involved part of this paper.
Its proof will be presented in Section \ref{sec-notriple}.
\begin{theorem} \label{th:momentwithoutTriple}
  Fix $\hat \b\in (0,1)$. Then there exists $c=c(\hat \beta)>0$ so that
  if \eqref{eq:condition20nr} holds then
 uniformly in $T\in \llbracket 1,N\rrbracket$ as $N\to\infty$,
\begin{equation} \label{eq:momentToT}
\sup_{X\in (\mathbb Z^2)^q}\DE_X^{\otimes q} \left[e^{\beta_N^2 \sum_{n=1}^T \sum_{1\leq i<j\leq q}  \mathbf{1}_{S_n^{i} = S_n^j}}\mathbf{1}_{G_T}\right] \leq e^{\lambda_{T,N}^2 \binom{q}{2} (1+cq^{-1/2} + o(1)) },
\end{equation}
where $\lambda_{T,N}$ is defined as
\begin{equation} \label{eq:lambdaT}
  \lambda_{T,N}^2(\hat \b) = \lambda_{T,N}^2=
  \log \frac 1 {1-{\hat \b}^2\frac{\log T}{\log N}}\ .
\end{equation}
\end{theorem}
Note that as soon as $q>9$, the expression in the left side of
\eqref{eq:momentToT} trivially vanishes
if $X={\bf 0}$. The $X$'s of interest
are those that allow for non-existence of
triple or more intersections.

Assuming Theorem \ref{th:momentwithoutTriple}, the proof of
Theorem \ref{th:main} is relatively straightforward. We will need
the preliminary results collected in the next subsection.

\subsection{A short time a priori estimate}
The following lemma is a variation on Khas'minskii's lemma \cite[p.8, Lemma 2.1]{Sznitman}.
\begin{lemma}\label{lem:modKhas} Let $\mathcal Z$ be the set of all nearest-neighbor walks on $\mathbb Z^2$, that is $Z\in\mathcal Z$ if $Z=(Z_i)_{i\in \mathbb N}$ where $Z_i\in\mathbb Z^2$ and $Z_{i+1}-Z_i \in \{\pm \mathbf e_j,j\leq d\}$ where $\mathbf e_j$ are the canonical vectors of $\mathbb Z^2$.  If for some $k\in \mathbb N$ and $\kappa\in \mathbb R$, one has
\begin{equation} \label{eq:condition_modKhas}
\eta = \sup_{Z\in \mathcal Z} \sup_{x\in \mathbb Z^2} \left(e^{\kappa^2}-1 \right) \DE_x\left[\sum_{n=1}^k \mathbf{1}_{S_n = Z_n}\right] <1,
\end{equation}
then
\begin{equation}
\sup_{Z\in \mathcal Z} \sup_{x\in \mathbb Z^2} \DE_{x} \left[e^{\kappa^2 \sum_{n=0}^k \mathbf{1}_{S_n=Z_n}} \right] \leq \frac{1}{1-\eta}.
\end{equation}
\end{lemma}
\begin{proof}

Let $\Lambda_2 = e^{\kappa^2}-1$. We have:
\begin{align*}
&\DE_{x} \left[e^{\kappa^2 \sum_{n=1}^k \mathbf{1}_{S_n=Z_n}} \right]
 = \DE_{x} \left[\prod_{n=1}^k \big(1+ \Lambda_2\mathbf{1}_{S_n=Z_n}\big)\right]\\
& = \sum_{p=0}^\infty \Lambda_2^p \sum_{1\leq n_1 < \dots<  n_p \leq k}\DE_{x} \left[ \prod_{i=1}^p \mathbf{1}_{S_{n_i}=Z_{n_i}}\right]\\
& = \sum_{p=0}^\infty \Lambda_2^p \sum_{1\leq n_1 < \dots < n_{p-1} \leq k} \DE_{x}
\left[
\prod_{i=1}^{p-1} \mathbf{1}_{S_{n_i}=Z_{n_i}} \DE_{S_{n_{p-1}}}
\left[\sum_{n=1}^{k-n_{p-1}} \mathbf{1}_{S_n = Z_{n+n_{p-1}}} \right]
\right]\\
& \stackrel{\eqref{eq:condition_modKhas}}{\leq}\sum_{p=0}^\infty \Lambda_2^{p-1}\, \eta \sum_{1\leq n_1 < \dots < n_{p-1} \leq k} \DE_{x}
\left[
\prod_{i=1}^{p-1} \mathbf{1}_{S_{n_i}=Z_{n_i}} 
\right]
 \leq \dots \leq
\sum_{p=0}^\infty  \eta^p = \frac{1}{1-\eta}.
\end{align*}
\end{proof}
The next lemma gives an a-priori rough estimate on the moments of
$W_k(\beta_N)$ $=$ $W_k(\beta_N,0)$ when $k$ is small. 
\begin{lemma} \label{lem:aprioriBound} Let $\hat \beta>0$.
Let $b_N>0$ be a deterministic sequence such that $b_N = o(\sqrt{\log N})$ as $N\to\infty$.
Assume that $q=O(\sqrt {\log N})>1$.
Then,
for all $k\leq e^{b_N}$,
\[
\IE[W_{k}(\beta_N)^q] =  \DE^{\otimes q}\left[ e^{\beta_N^2 \sum_{1\leq i<j\leq q} \sum_{n=1}^k \mathbf{1}_{S_n^{i} = S_n^j}}\right] \leq e^{\frac{1}{\pi}(1+\varepsilon_N)  q^2 \beta_N^2 \log (k+1)},\]
for $\varepsilon_N = \varepsilon_N(\hat \beta)\to 0$ as $N\to\infty$.
\end{lemma}
\begin{proof}
 Let $N^{i,j}_k=\sum_{n=1}^k \mathbf{1}_{S_n^{i} = S_n^j}$.
By H\"older's inequality,
we find that
\begin{align*}
\IE[W_{k}(\beta_N)^q] & \leq \DE^{\otimes q}\left[ e^{\frac{q\beta_N^2}{2} \sum_{ 1 < j\leq q}  N^{1,j}_k} \right]^{q/q}
 = \DE \left[\DE^{\otimes 2}\left[ e^{\frac{q\beta_N^2}{2} N^{1,2}_k} \middle | S^1\right]^{q-1}\right],
\end{align*}
by independence of the $(N_{1,j})_{1<j}$ conditioned on $S^1$.
We now estimate the above conditional expectation using Lemma \ref{lem:modKhas}.
Let $\kappa^2= q\beta_N^2/2\to 0$ and $\eta$ be as in \eqref{eq:condition_modKhas}. For any $Z\in \mathcal Z$ and $y\in \mathbb Z^2$,
\[\DE_y\left[\sum_{n=1}^k \mathbf{1}_{S_n = Z_n}\right] \leq \sum_{n=1}^k \sup_{x} p_n(x),
\]
where, see Appendix \ref{sec-pnstar} for an elementary proof,
\begin{equation}\label{eq:pnstar}
\forall n\geq 1: \quad  \sup_x p_n(x) =:p_n^{\star} \leq \frac{2}{\pi n}.
\end{equation}
Thus,  $\eta \leq \frac{1}{\pi} (1+o(1)) q\beta_N^2 \log (k+1) \to 0$, uniformly for $k\leq e^{b_N} $ as $N\to\infty$.
Lemma \ref{lem:modKhas} then yields that for such $k$'s,
\[
\IE[W_{k}(\beta_N)^q] \leq \left(\frac{1}{1-\frac{1}{\pi}(1+o(1))q\beta_N^2 \log (k+1)}
\right)^{q-1} = e^{\frac{1}{\pi}(1+o(1)) q^2 \beta_N^2 \log (k+1)}.\]
\end{proof}

\subsection{Proof of Theorem \ref{th:main}.}
As a first step, we will prove that
\begin{equation} \label{eq:firstStepQmoment}
\IE[W_N^q] \leq e^{{\binom{q}{2} \lambda^2}(1+cq^{-1/2}+\varepsilon_N)},
\end{equation}
where $c=c(\hat \beta)> 0$ 
and $\varepsilon_N=\varepsilon_N(\hat \beta)\to 0$ as $N\to\infty$.

As a second step, we improve the bound in case $q$ is bounded and thus complete the proof for general
$q(N)$  assuming only condition \eqref{eq:condition20nr}, by a diagonalization argument.

Recall the definitions of $\lambda_{k,N}$ in \eqref{eq:lambdaT} and
that
$\lambda=\lambda_{N,N}(\hat \beta)$.
By standard convexity arguments, we note that $x\leq \log(\frac{1}{1-x}) \leq \frac{x}{1-x}$ for all $x\in [0,1)$;
hence for all $a>1$ and $\hat \beta < 1$ such that $a\hat \beta^2 < 1$,
 \begin{equation} \label{eq:bounds_lambda}
 \forall k\leq N:\quad  a \hat \beta^2 \frac{\log k}{\log N}\leq \lambda_{k,N} (\sqrt a \hat \beta)^2 \leq\frac{a \hat \beta^2}{1-a\hat \beta^2} \frac{\log k}{\log N}.
 \end{equation}

 Now, let
 \[I_{s,t}=\beta_N^2\sum_{n=s+1}^t \sum_{i<j\leq q} \mathbf{1}_{S_n^{i} = S_n^j} \quad  \text{and} \quad
 I_k = I_{0,k},\] and define
\begin{equation}
M(X) := \DE_X^{\otimes q}\left[e^{I_{N}}\right] \quad \text{and} \quad  M=\sup_{X\in (\mathbb Z^2)^q} M(X).
\end{equation} By \eqref{eq:momentFormula}, it is enough to have a bound on $M(\mathbf 0)$. In fact what we will give is a bound on $M$. To do so, we let $T=T_N>0$ such that $\log T =o(\sqrt{\log N})$ and introduce the event
\[\tau_{T} := \inf \{n > T : F_n\cup K_n \, \mbox{\rm occurs}\}.\]
We then decompose $M(X)$ as follows:
\begin{equation} \label{eq:defAB}
M(X)  = \DE_X^{\otimes q}\left[e^{I_{N}} \mathbf{1}_{\tau_T \leq N}\right] + \DE_X^{\otimes q}\left[e^{I_{N}} \mathbf{1}_{\tau_T >N}\right] =: A(X)+ B(X).
\end{equation}
We start by bounding $B(X)$ from above. Let $c$ be as in Theorem \ref{th:momentwithoutTriple}. By Markov's property,
\begin{align} 
 \sup_{X\in (\mathbb Z^{2})^q} B(X) & \leq \sup_{X\in (\mathbb Z^{2})^q} \DE^{\otimes q}_X[e^{I_{T}}] \sup_{Y\in (\mathbb Z^{2})^q} \DE_Y^{\otimes q} [e^{I_{N-T}} \mathbf{1}_{\tau_0 > N-T}] \nonumber\\
  & \leq e^{\frac{1}{\pi}(1+\varepsilon_N)q^2 \beta_N^2 \log T} e^{{\binom{q}{2}\lambda_{N-T,N}^2}
 (1+cq^{-1/2}+o(1))} \nonumber\\
 & \leq  e^{{\binom{q}{2} \lambda^2}
 (1+cq^{-1/2}+o(1))},\label{eq:MarkovOnB}
\end{align}
where in the second inequality, we used Lemma \ref{lem:aprioriBound} and
Theorem \ref{th:momentwithoutTriple} and in the last inequality, we used that $\beta_N^2\log T$ vanishes as $N\to\infty$ and that $\lambda_{N-T,N}^2 <\lambda^2$. 

We will now deal with $A(X)$ and show that
\begin{equation} \label{eq:boundFinalAX}
\sup_{X\in (\mathbb Z^2)^q} A(X) \leq M \varepsilon_N,
\end{equation}
with $\varepsilon_N\to 0$.
This, together with \eqref{eq:defAB} and \eqref{eq:MarkovOnB} implies that
\[M(1-\varepsilon_N) \leq  e^{{\binom{q}{2} \lambda^2}(1+cq^{-1/2}+o(1))},
\]
which entails \eqref{eq:firstStepQmoment}.

Toward the proof of \eqref{eq:boundFinalAX},
we first use Markov's property to obtain that
\begin{equation*}A(X)  = \sum_{k=T}^N \DE_X^{\otimes q}[e^{I_{k}+I_{k ,N}} \mathbf{1}_{\tau_T = k}] \leq M  \sum_{k=T}^N \DE_X^{\otimes q}[e^{I_{k}} \mathbf{1}_{\tau_T = k}] .
\end{equation*}
In what follows, for $\mathcal T\subset \mathbb N$ we use the phrase
 "no triple+ in  $\mathcal T$" to denote the event
 $(\cup_{n\in \mathcal T} (F_n \cup K_n))^\complement$. Similarly, 
 for $\mathcal I\subset \llbracket 1,q\rrbracket$ we use the phrase
 "no triple+ in $\mathcal T$ for particles of $\mathcal I$"  to denote the event
 $(\cup_{n\in \mathcal T} (F_n^{\mathcal I} \cup K_n^{\mathcal I}))^\complement$ 
 where $F_n^{\mathcal I}$ and $G_n^{\mathcal I}$ are defined as $F_n$ and $K_n$ but with $\llbracket 1,q\rrbracket$ replaced by $\mathcal I$. 
We then decompose over which event, $F_n$ or $K_n$, occurred at $\tau_T$, and then over which particles
participated in the event:
\begin{equation} \label{eq:AXbound2}
\begin{aligned}
  A(X)&\leq M \sum_{\balpha,\bbeta,\bgamma \leq q} \sum_{k=T}^N \DE_X^{\otimes q}\left[e^{I_{k}} \mathbf{1}_{\text{no triple+ in $\llbracket T, k-1\rrbracket$}}\, \mathbf{1}_{S^\balpha_k = S^\bbeta_k = S^\bgamma_k}\right]\\
  & + M\sum_{(\balpha<\bbeta)\neq(\bgamma<\bdelta)} \sum_{k=T}^N \DE_X^{\otimes q}\left[e^{I_{k}} \mathbf{1}_{\text{no triple+ in $\llbracket T, k-1\rrbracket$}}\, \mathbf{1}_{S^\balpha_k = S^\bbeta_k,   S^\bgamma_k=S^\bdelta_k}\right],\\
&=: M(A_1(X)+A_2(X)).
\end{aligned}
\end{equation}
By \eqref{eq:AXbound2}, it is enough to prove that $A_1(X)$ and $A_2(X)$ vanish uniformly in $X$ as $N\to\infty$ in order to obtain \eqref{eq:boundFinalAX}.
We next show that $A_1(X)$ vanishes, the argument for $A_2(X)$ is similar.
We bound $I_k$ by
\[I_{k} \leq J_{k} + J^\balpha_{k} +J^\bbeta_{k} +J^\bgamma_{k} ,\]
where
\begin{gather*}J_{k} = \beta_N^2 \sum_{n=1}^k \sum_{\substack{i<j\leq q\\i,j\notin\{\balpha,\bbeta,\bgamma\}}} \mathbf{1}_{S^i_n=S^j_n} \quad\text{and}\quad
J^{i_0}_{k} = \beta_N^2\sum_{n=1}^k \sum_{j\in\llbracket 1,q\rrbracket \setminus \{i_0\} } \mathbf{1}_{S^{i_0}_n=S^j_n}.
\end{gather*}
If we let $\frac{1}{a} + \frac{3}{b} = 1$ with $1<a\leq 2$ and $1<b$, we have \begin{align}
  &\DE_X^{\otimes q}
  \left[e^{I_{k}} \mathbf{1}_{\text{no triple+ in $\llbracket T, k-1\rrbracket$}}\, \mathbf{1}_{S^\balpha_k = S^\bbeta_k = S^\bgamma_k}\right] \nonumber\\
  &\leq \DE_X^{\otimes q}\left[e^{aJ_{k}} \mathbf{1}_{\text{no triple+ in $\llbracket T, k-1\rrbracket$ for particles of $\llbracket1,q\rrbracket \setminus \{\balpha,\bbeta,\bgamma\}$}}\, \mathbf{1}_{S^\balpha_k = S^\bbeta_k = S^\bgamma_k}\right]^{1/a} \label{eq:boundNotripleHolder}\\
  &\times \prod_{i_0\in \{\balpha,\bbeta,\bgamma\}} \DE_X^{\otimes q}\left[e^{bJ^{i_0}_{k}} \mathbf{1}_{S^\balpha_k = S^\bbeta_k = S^\bgamma_k}\right]^{1/b}. \label{eq:boundNotripleHolder2}
\end{align}

We treat separately the two quantities \eqref{eq:boundNotripleHolder} and \eqref{eq:boundNotripleHolder2}. Before doing so,
we specify  our choice of $a,b$ and $\hat \beta$. We assume that $\hat \beta^2 < 1/72$ and $a<3/2$, with $a$ close enough to $3/2$ (and so $b$ close to $9$) in such a way that
\begin{equation} \label{eq:conditions}
\begin{aligned}
&(i)\ 8b\hat \beta^2 < 1, \quad (ii)\ \limsup_{N\to\infty} \frac{1}{\pi} q^2 \beta_N^2 =:\rho_0 < 1/a \quad \text{and} \\
&\quad (iii)\ \limsup_{N\to \infty} \frac{\hat \beta^2}{1-a\hat \beta^2} \frac{\binom q 2}{\log N}(1+cq^{-1/2}) =:\rho_1 < 1/a.
\end{aligned}
\end{equation}
Note that (ii) and (iii) are assured to hold for $a$ close enough to $3/2$ thanks to the assumption \eqref{eq:condition20nr} which implies that $\limsup_N \pi^{-1}q^2 \beta_N^2 \leq \frac{2}{3}$ since $\beta_N^2\sim \pi \hat \beta^2/\log N$. We chose $\hat \beta^2 < 1/72$ to allow (i).

We first bound \eqref{eq:boundNotripleHolder}.
If $k \leq e^{(\log N)^{1/3}}$, then, using that $J_k$ does not
depend on the $\bar \alpha, \bar \beta,\bar \gamma$ particles,  \eqref{eq:boundNotripleHolder} is bounded by
\begin{align*}
  &\DE_X^{\otimes q}\left[e^{aJ_{k}} \right]^{1/a} \DP_{(x^{\bar\alpha},x^{\bar \beta},x^{\bar \gamma})}^{\otimes 3}\left(S^\balpha_k = S^\bbeta_k = S^\bgamma_k\right)^{1/a}\leq C e^{\frac{1}{\pi} (1+\varepsilon_N) q^2 a \beta_N^2 (\log (k+1))/a} k^{-2/a},
\end{align*}
for some $c>0$ and uniformly in $X\in (\mathbb Z^2)^q$, where we have used in the inequality Lemma \ref{lem:aprioriBound} 
and that  $\sum_{x} p_{k}(x)^3 \leq (p_k^\star)^2 \leq k^{-2}$ by \eqref{eq:pnstar}.
For $k \geq e^{(\log N)^{1/3}}$, we rely on \eqref{eq:momentToT} to bound 
\eqref{eq:boundNotripleHolder} by
\begin{align*}
  & \DP_{(x^{\bar \alpha},x^{\bar \beta},x^{\bar \gamma})}^{\otimes 3}\left(S^\balpha_k = S^\bbeta_k = S^\bgamma_k\right)^{1/a}\\
&\qquad \times \DE_X^{\otimes q}\left[e^{aJ_{k}} \mathbf{1}_{\text{no triple+ in $\llbracket T, k-1\rrbracket$ for particles in $\llbracket1,q\rrbracket \setminus \{\balpha,\bbeta,\bgamma\}$}}
\,\right]^{1/a} \\
& \leq   C k^{-2/a} \left( \DE_X^{\otimes q}\left[e^{aJ_{T}}\right] \sup_{Y} \DE_Y^{\otimes (q-3)}\left[e^{aJ_{k-T-1}} \mathbf{1}_{\text{no triple+ in $\llbracket 1, k-T-1\rrbracket$
}}\right]\right)^{1/a}
\\
&\leq C e^{\frac{1}{\pi}(1+\varepsilon_N) q^2 a \beta_N^2 (\log T)/a} e^{{\binom{q}{2}(1+cq^{-1/2}+\varepsilon'_N) \lambda_{k,N}^2(\sqrt a \hat \beta)  /a}} k^{-2/a},
\end{align*}
where we have used that $J_{k-T}\leq C + J_{k-T-1}$ by \eqref{eq:condition20nr}.
For \eqref{eq:boundNotripleHolder2}, we apply the Cauchy-Schwarz inequality to find that
\begin{equation} \label{eq:CS3}
  \DE_X^{\otimes q}\left[e^{bJ^{i_0}_{k}} \mathbf{1}_{S^\balpha_k = S^\bbeta_k = S^\bgamma_k}\right]^{1/b}\leq \DE_X^{\otimes q}\left[e^{2bJ^{i_0}_{k}} \right]^{1/2b}  k^{-1/b},
\end{equation}
where we again used that $\sum_{x} p_{k}(x)^3 \leq (p_k^\star)^2 \leq k^{-2}$ by \eqref{eq:pnstar}.
Now observe that by conditioning on $S^{i_0}$, we have
\begin{align*}
  \DE_X^{\otimes q}\left[e^{2bJ^{i_0}_{k}}\right] & \leq \sup_{y\in \mathbb Z^2} \DE^{S^1}_{y} \left[\sup_{x\in \mathbb Z^2} \DE^{S^2}_{x} \left[e^{2b \beta_N^2\sum_{n=1}^k \mathbf{1}_{S_k^{1} = S_k^2}} \right]^{q-1} \right],
\end{align*}
where uniformly on all nearest neighbor walks $Z\in {\mathcal Z}$, 
\[ \left(e^{2b\beta_N^2} - 1\right) \sup_{x\in \mathbb Z^2} \DE_{x}\sum_{n=1}^k \mathbf{1}_{S_n = Z_n}\leq 4(1+o(1))b  \hat \beta^2 \frac{\log (k+1)}{\log N}
\leq 8b  \hat \beta^2 \frac{\log (k+1)}{\log N}\]
for all $N$ large,
because $\beta_N^2\sim \pi \hat \beta^2/\log N$ and 
$\sup_x p_n(x)\leq 2/(\pi n)$, see \eqref{eq-RNas} and \eqref{eq:pnstar}.
Hence by
Lemma \ref{lem:modKhas} with   \eqref{eq:conditions}-(i),
 \begin{align*}
   \sup_{X\in (\mathbb Z^2)^q}\DE_X^{\otimes q}\left[e^{bJ^{i_0}_{k}} \mathbf{1}_{S^\balpha_k = S^\bbeta_k = S^\bgamma_k}\right]^{1/b}& \leq \left(\frac{1}{1-8 b\hat \beta^2 \frac{\log (k+1)}{\log N}}\right)^{(q-1)/2b} k^{-1/b}\\
 & \leq e^{c \frac{\log (k+1)}{\sqrt{\log N}}} k^{-1/b},
 \end{align*}
 for some universal constant $c>0$, using \eqref{eq:condition20nr}.

We thus find that
\begin{equation}
\label{eq:3rdboundAX}
 \begin{aligned}
& \sup_{X\in (\mathbb Z^2)^q} A_1(X) \leq  q^3   \sum_{k=T}^{\lfloor e^{(\log N)^{1/3}} \rfloor}
  e^{\frac{1}{\pi}(1+\varepsilon_N) q^2 \beta_N^2 \log (k+1)} k^{-2/a}   e^{3c \frac{\log (k+1)}{\sqrt{\log N}}}  k^{-3/b}\\
  &+ C q^3 e^{\frac{1}{\pi}(1+\varepsilon_N) q^2  \beta_N^2 \log T} \times\\
  &\times  \sum_{k=\lfloor e^{(\log N)^{1/3}} \rfloor}^N
  e^{{\lambda_{k,N}^2(\sqrt a \hat \beta) \binom{q}{2}(1+cq^{-1/2}+\varepsilon'_N)/a}} k^{-2/a}   e^{3c \frac{\log (k+1)}{\sqrt{\log N}}}  k^{-3/b}.
 \end{aligned}
\end{equation}
We now prove that the two terms on the right-hand side of \eqref{eq:3rdboundAX} vanish as $N\to\infty$.
By \eqref{eq:conditions}-(ii),(iii), 
there exists $\delta>0$ 
satisfying $\delta< 1/a-\max (\rho_0,\rho_1)$, and therefore
the first sum in the right-hand side of \eqref{eq:3rdboundAX} can be bounded by
\begin{align*}
 q^3 \sum_{k=T}^{\lfloor e^{(\log N)^{1/3}} \rfloor}   k^{-1-\delta} \leq Cq^3 T^{-\delta},
\end{align*}
for $N$ large enough.
 Hence, we can set $T = \lfloor e^{(\log N)^{1/4}} \rfloor $ (which satisfies $\log T = o(\sqrt{\log N})$), so that $q^3 T^{-\delta}\to 0$ as $N\to\infty$.
Relying on \eqref{eq:bounds_lambda}, the second sum {in \eqref{eq:3rdboundAX}} is bounded by
\begin{align*}
&C q^3 e^{c '  \log T} \sum_{k=\lfloor e^{(\log N)^{1/3}} \rfloor}^N
e^{{\frac{\hat\beta^2}{1-a\hat \beta^2}\frac{\binom{q}{2}(1+cq^{-1/2}+\varepsilon'_N)}{\log N}}
{\log (k+1)}}    e^{3c \frac{\log (k+1)}{\sqrt{\log N}}}  k^{-1-1/a}\\
  & \leq C q^3  e^{c' \log T}  \sum_{k=\lfloor e^{(\log N)^{1/3}} \rfloor}^N k^{-1-\delta/2}
  \leq C q^3 e^{-\frac{\delta}{2} (\log N)^{1/3} + c  (\log N)^{1/4}},
\end{align*}
for some $c'>0$ (recall that  $\delta < 1/a - \rho_1$).
Then the quantity in the last line vanishes as $N\to\infty$.   (Note that we decomposed the sum for $k\geq e^{(\log)^{1/3}}$ and let $\log T = (\log N)^{1/4}$ to ensure that $(\log N)^{1/3}\gg (\log N)^{1/4}$ in the last display). By \eqref{eq:3rdboundAX} we have thus proven that $\lim_{N\to\infty}\sup_{X\in (\mathbb Z^2)^q} A_1(X)=  0$. 

When dealing with $A_2$, we have to use H\"{o}lder's inequality as in \eqref{eq:boundNotripleHolder}, \eqref{eq:boundNotripleHolder2} with $4$ particles instead of $3$, so in this case we can choose $a\sim 3/2$ and $b\sim 12$, and the condition (i) in \eqref{eq:conditions} is satisfied with the restriction $\hat \beta^2 < 1/96$. The rest of the argument follows the same line as for $A_1$ 
and we get that $\lim_{N\to\infty}\sup_{X\in (\mathbb Z^2)^q} A_2(X)=  0$. 

 As a result, we have shown that \eqref{eq:firstStepQmoment} holds. This proves \eqref{eq:estimate} when $q\to\infty$.
 When $q=q_0$, 
\eqref{eq:firstStepQmoment} yields that $W_N$ is bounded in any $L^p$, $p>1$. This fact combined with \eqref{eq:pointwiseGaussianCV} implies the convergence \eqref{eq-convergencepol} for all fixed $q$, which implies that \eqref{eq:estimate} holds in the case $q=q_0$ as well.

We now turn to the general case, where we only assume that $q(N)$  satisfies \eqref{eq:condition20nr}. Suppose that \eqref{eq:estimate} does not hold, so that we can find $\varepsilon_0>0$ and a subsequence $q_N'=q(\varphi(N))$ such that
\begin{equation} \label{eq:absurd}
\forall N\in\mathbb N,\quad \IE W_N^{q_N'} > e^{\lambda^2 \binom{q_N'}{2}(1+\varepsilon_0)}.
\end{equation}

 One can distinguish two cases. If $q_N'$ is bounded, then up to extracting a sub-sequence, we can suppose that $q_N'$ converges to some $q_0\geq 2$. Then, one can check that by \eqref{eq:pointwiseGaussianCV}, we must have $\IE W_N^{q_N'} \to e^{\lambda^2 \binom{q_0}{2}}$ (for example, using Skorokhod's representation theorem and Vitali's convergence theorem with the fact that $W_N$ is bounded in any $L^p$). But this is impossible by \eqref{eq:absurd}. On the other hand, if $q'_N$ is not bounded, up to extracting a subsequence we can suppose that $q'_N\to\infty$. But then \eqref{eq:absurd} cannot be true because \eqref{eq:estimate} holds with $q=q'_N\to\infty$. Therefore \eqref{eq:estimate} must hold for any sequence $q(N)$ that satisfies \eqref{eq:condition20nr}.
\qed

\subsection{On Remark \ref{rem-1.2}}
\label{subsec-rem1.2}
We describe the changes needed for obtaining the claim in Remark
\ref{rem-1.2}. Recall the definitions of $F_n$ and  $K_n$,
see \eqref{eq-Fn} and \eqref{eq-Kn}, and \eqref{eq:momentFormula}.
Set
\begin{align*}
  &A_N=\sum_{n=1}^N \mathbf{1}_{(F_n\cup K_n)^\complement}
  \sum_{1\leq i<j\leq q} \mathbf{1}_{S_n^{i} = S_n^j},\\
  &B_N =   \sum_{n=1}^N \mathbf{1}_{F_n} \sum_{1\leq i<j\leq q} \mathbf{1}_{S_n^{i} = S_n^j}, \quad C_N=
  \sum_{n=1}^N \mathbf{1}_{K_n} \sum_{1\leq i<j\leq q} \mathbf{1}_{S_n^{i} = S_n^j}
\end{align*}
  Note that for any $u_N\geq 1$, we can check that
  $\DE^{\otimes q}_X \left[e^{u_N \beta_N^2 A_N}\right]$ is bounded above by $\Psi_{N,q}(X)$ of \eqref{eq:defPsiN}
  with $T=N$ and $\beta_N$
  replaced by $\beta_N u_N$.
Using H\"older's inequality
it is enough to show (together with
the proof of Theorem \ref{th:momentwithoutTriple}, which actually controls
$\sup_{X}\Psi_{N,q}(X)$) that
for any $\hat \beta<1$ and $q_N=o(\log N/\log \log N)$, there exist
$v_N\to\infty$ so that
\begin{equation} \label{eq:Holder}
 \sup_X \DE^{\otimes q}_X \left[e^{v_N \beta_N^2 B_N}\right]^{1/v_N} \to_{N\to\infty} 1,
  \quad \sup_X \DE^{\otimes q}_X
  \left[e^{v_N \beta_N^2 C_N}\right]^{1/v_N} \to_{N\to\infty} 1.
\end{equation}
We sketch the proof of the first limit in \eqref{eq:Holder}, the proof of the second is similar.
By Corollary \ref{cor:discrete_Khas} (applied on the space of $q$-tuples of path, 
with $f(Y_n)=v_N \beta_N^2\sum_{1\leq i<j \leq q} 
{\bf 1}_{S^i_n=S^j_n} {\bf 1}_{F_n}$),
it suffices to show that \begin{equation} \label{eq:khasCondition}
 \limsup_{N\in \mathbb N} \sup_{X \in \mathbb Z^q}  \DE_X^{\otimes q}[v_N\b_N^2 B_N] = 0.
\end{equation}
To see
\eqref{eq:khasCondition}, fix
$K\in \llbracket 1,N\rrbracket$. By \eqref{eq:pnstar},
we have that
  \begin{align}
  \DE_X^{\otimes q} \sum_{n=1}^N \mathbf{1}_{F_n} \sum_{1\leq i<j\leq q} \mathbf{1}_{S_n^{i} = S_n^j}
  \leq \sum_{n=1}^K \frac{q(q-1)}{2} \frac{C}{n} + \sum_{n=K+1}^N \sum_{1\leq i<j\leq q}  \DE_X^{\otimes q} \mathbf{1}_{F_n}  \mathbf{1}_{S_n^{i} = S_n^j}.\label{eq:boundRemoveTripleStep0}
 \end{align}
 For $i<j\leq q$ and $r\in \{0,1,2\}$, further denote
\[ F_n^{i,j;r} = \{\exists (\bar\alpha,\bar\beta,\bar\gamma) :
\bar \alpha< \bar\beta<\bar\gamma \leq q,  S_n^{\bar \alpha} = S_n^{\bar \beta}
= S_n^{\bar \gamma}, |\{\bar\alpha,\bar\beta,\bar\gamma\}\cap \{i,j\}| = r\}.\]
We have that
\[
\sum_{1\leq i<j\leq q}  \DE_X^{\otimes q} \mathbf{1}_{F_n}  \mathbf{1}_{S_n^{i} = S_n^j} = \sum_{1\leq i<j\leq q}  \DE_X^{\otimes q} \sum_{r=0}^2 \mathbf{1}_{F_n^{i,j;r}}  \mathbf{1}_{S_n^{i} = S_n^j}.
\]
We first focus on the term $r=0$. By independence, \eqref{eq:pnstar}
and the union bound,
\begin{align*}\sum_{1\leq i<j\leq q}  \DE_X^{\otimes q} \mathbf{1}_{F_n^{i,j;0}}  \mathbf{1}_{S_n^{i} = S_n^j} & \leq \sum_{1\leq i<j\leq q}  \frac{C}{n} \sum_{\bar\alpha<\bar\beta<\bar\gamma\leq q} \sup_{x_i\in \mathbb Z^2} \sum_{y\in \mathbb Z^2} \prod_{i=1}^3 \DP_{x_i}(S_n = y)
 \leq  \frac{Cq^5}{n^3}\,.
\end{align*}
When $r=1$, the condition in the indicator function becomes that there exist $\bar \alpha < \bar\beta \leq q$ such that $S_n^i = S_n^j = S_n^{\bar \alpha} = S_n^{\bar \beta}$. Hence, the term for $r=1$ is bounded by
\[
\sum_{1\leq i<j\leq q}  \sum_{\bar \alpha<\bar \beta\leq q}
\sup_{x_i\in {\mathbb Z}^2}\sum_{y\in \mathbb Z^2} \prod_{i=1}^4\DP_{x_i}(S_n = y)\\
 \leq  \frac{Cq^4}{n^3}\,.
\]
Similarly, we can bound the term for $r=2$ by a constant times $q^3/n^2$. Using \eqref{eq:boundRemoveTripleStep0}, we find that for all $K \in \llbracket 1,N\rrbracket$,
\begin{equation} \sup_{X \in \mathbb Z^q}  \DE_X^{\otimes q}[v_N\b_N^2 B_N] \leq \frac {Cv_N \hat \beta^2} {\log N} \left(  \frac{q(q-1)}{2}  \log K + \frac{q^5}{K^2} + \frac{q^4}{K^2} +  \frac{q^3}{K}\right).
\end{equation}
For $K=\left \lfloor (\log N)^{3/4} \right\rfloor$, and $q^2 = o(\log N / \log \log N)$, we find that \eqref{eq:khasCondition} holds with
a well-chosen $v_N\to\infty$.

\section{No triple intersections - Proof of Theorem \ref{th:momentwithoutTriple}}
\label{sec-notriple}
Recall that $T\in \llbracket 1,N\rrbracket$.
For compactness of notation in the rest of the paper, set 
\begin{equation}\label{eq:defSigmaN} \sigma_N^2= \sigma_N^2(\hat \beta) = e^{\b_N^2}-1.
\end{equation}
By \eqref{eq-RNas},
there exist
$\delta_N=\delta(N,\hat \b)$ and $\delta'_N=\delta'(N,\hat \b)$ that vanish as $N\to\infty$ such that
\begin{equation} \label{eq:LLTsigmaN}
  \sigma_N^2 = \frac{\hat \beta^2}{R_N}(1+\delta_N) = \frac{ \pi \hat \beta^2}{\log N} (1+\delta'_N).
\end{equation}
\subsection{Expansion in chaos} \label{sec:chaos} In this section, we show that the moment without triple intersections can be bounded by a rather simple expansion. Introduce the following notation:
for $\mathbf n = (n_0,n_1,\dots,n_k)$ and $\mathbf x = (x_0,x_1,\dots,x_k)$, let $p_{\mathbf n,\mathbf x} = \prod_{i=1}^k p(n_{i}-n_{i-1},x_i-x_{i-1})$. 
Let  $\mathcal C_q = \{(i,j)\in \llbracket 1,q\rrbracket^2 : i<j\}$.
\begin{proposition} \label{prop:decInChaos}
For all $X=(x_0^1,\dots,x_0^q)\in (\mathbb Z^2)^q$, we have
\begin{equation} \label{eq:lessthanPsi}
\DE_X^{\otimes q} \left[e^{\beta_N^2 \sum_{n=1}^T \sum_{1\leq i<j\leq q}  \mathbf{1}_{S_n^{i} = S_n^j}}\mathbf{1}_{G_T}\right] \leq \Psi_{N,q}(X),
\end{equation}
where
\begin{equation}
\begin{aligned}
 &\Psi_{N,q}(X) =\\
 & \sum_{k=0}^\infty \sigma_N^{2k} \sum_{\substack{1\leq n_1<\dots<n_k\leq T\\
 (i_r,j_r)_{r\leq k}\in \mathcal C_q^k\\ \mathbf x^1 \in (\mathbb Z^2)^{k},\dots,\mathbf x^q \in (\mathbb Z^2)^{k}}} \prod_{r=1}^k \mathbf{1}_{x_r^{i_r} = x_r^{j_r}} \prod_{i=1}^q p_{(0,n_1,\dots,n_k),(x_0^i,\mathbf x^i)},
 \end{aligned}
 \label{eq:defPsiN}
\end{equation}
where we recall \eqref{eq:defSigmaN} for the definition of $\sigma_N$.
\end{proposition}
\noindent
(By convention, here and throughout the paper, the term $k=0$ in sums as \eqref{eq:defPsiN} equals $1$.)
\begin{proof}
For brevity, we write $G$ for $G_T$.
For $X=(x_0^1,\dots,x_0^q)\in (\mathbb Z^2)^q$, using the identity $e^{\beta_N^2  \mathbf{1}_{S_n^{i} = S_n^j}} -1 =  \sigma_N^2 \,  \mathbf{1}_{S_n^{i} = S_n^j}$,
\begin{align*}
M_{N,q}^{\text{no triple}}(X) &:= \DE_X^{\otimes q} \left[e^{\beta_N^2 \sum_{n=1}^T \sum_{1\leq i<j\leq q}  \mathbf{1}_{S_n^{i} = S_n^j}}\mathbf{1}_{G}\right] \\
& = \DE_X^{\otimes q} \left[\prod_{n\in\llbracket 1,T\rrbracket,(i,j)\in\mathcal C_q} \left(1+\sigma_N^2 \,  \mathbf{1}_{S_n^{i} = S_n^j} \right)\mathbf{1}_{G}\right].
\end{align*}

Expand the last product to obtain that:
\begin{equation*}
M_{N,q}^{\text{no triple}}(X)  = \sum_{k=0}^\infty \sigma _N^{2k} \sum_{{\substack{(n_r,i_r,j_r)\in \llbracket 1,T\rrbracket\times \mathcal C_q, r=1,\dots,k\\(n_1,i_1,j_1)<\dots<(n_k,i_k,j_k)}}} \DE_X^{\otimes q} \left[\prod_{r=1}^k \mathbf{1}_{S_{n_r}^{i_r} = S_{n_r}^{j_r}}\mathbf{1}_{G}\right],
\end{equation*}
where we have used the lexicographic ordering on 3-tuples $(n,i,j)$.
Since there are no triple or more particle intersections on the event $G$, the sum above can be restricted to 3-tuples $(n_r,i_r,j_r)_{r\leq k}$ such that $n_r < n_{r+1}$ for all $r < k$. Hence,
\begin{align*}
M_{N,q}^{\text{no triple}}(X)& = \sum_{k=0}^\infty \sigma_N^{2k} \sum_{1\leq n_1<\dots<n_k\leq T,(i_r,j_r)_{r\leq k}\in \mathcal C_q^k} \DE_X^{\otimes q} \left[\prod_{r=1}^k  \mathbf{1}_{S_{n_r}^{i_r} = S_{n_r}^{j_r}} \mathbf{1}_{G}\right]\\
& \leq \Psi_{N,q}(X),
\end{align*}
where $\Psi$ is defined in \eqref{eq:defPsiN}, and
where we have bounded $\mathbf{1}_{G}$ by $1$ in the inequality.
\end{proof}
\subsection{Decomposition in two-particle intersections}
In this section, we rewrite $\Psi_{N,q}$  in terms of successive two-particle interactions.  We generalize a decomposition used in \cite[Section 5.1]{CaSuZyCrit18} that was restricted to a third moment computation ($q=3$).
The following notation is borrowed from their paper. Let
\begin{equation}
U_N(n,x) := \begin{cases}
\sigma_N^2 \DE_{0}^{\otimes 2}\left[ e^{ \b_N^2\sum_{l=1}^{n-1} \mathbf 1_{S^1_l = S^2_l}} \mathbf{1}_{S^1_n= S^2_n=x} \right]  & \text{if } n\geq 1,\\
 \mathbf{1}_{x=0} & \text{if } n=0,
\end{cases}
\end{equation}
and
\begin{equation}  \label{eq:def_UN}
\begin{aligned}
U_N(n) & := \sum_{z \in\mathbb Z^2} U_N(n,z)
= \begin{cases}
\sigma_N^2 \DE_{0}^{\otimes 2}\left[ e^{ \b_N^2\sum_{l=1}^{n-1} \mathbf 1_{S^1_l = S^2_l}} \mathbf{1}_{S^1_n= S^2_n} \right] & \text{if } n\geq 1,\\
1 & \text{if } n=0.
\end{cases}
\end{aligned}
\end{equation}
Observe that, by the identity $e^{\beta_N^2  \mathbf{1}_{S_l^{1} = S_l^2}} -1 =  \sigma_N^2 \,  \mathbf{1}_{S_l^{1} = S_l^2}$, one has for all $n\geq 1$,
\begin{equation}\label{eq:chaosSimple}
\begin{aligned}
&\DE_{0}^{\otimes 2}\left[ e^{ \b_N^2\sum_{l=1}^{n-1} \mathbf 1_{S^1_l = S^2_l}} \mathbf{1}_{S^1_n= S^2_n=x} \right]  = \DE_{0}^{\otimes 2}\left[ \prod_{l=1}^{n-1} \left(1+\sigma_N^2\mathbf 1_{S^1_l = S^2_l}\right) \mathbf{1}_{S^1_n= S^2_n=x} \right]\\
& = \sum_{k=0}^\infty \sigma_N^{2k} \sum_{n_0=0 < n_1 <\dots < n_{k}< n}  \DE_0^{\otimes 2} \left[ \prod_{r=1}^{k} \mathbf{1}_{S_{n_r}^1 = S_{n_r}^2} \mathbf{1}_{S_{n}^1 = S_{n}^2 = x} \right].
\end{aligned}
\end{equation}
Hence for all $n\geq 1$:
\begin{equation} \label{eq:chaos2}
U_N(n,x) = \sigma_N^2 \sum_{k=0}^\infty \sigma_N^{2k} \sum_{\substack{n_0=0 < n_1 <\dots < n_{k}< n = n_{k+1}  \\x_0=0 ,x_1,\dots,x_{k}\in \mathbb Z^2,x_{k+1}=x}}  \prod_{r=1}^{k+1} p_{n_{r}-n_{r-1}}(x_r-x_{r-1})^2.
\end{equation}

Now, in the sum in \eqref{eq:defPsiN}, we observe that (only) two particles interact at given times $(n_1<\dots< n_k)$. So we define $a_1=n_1$ and $b_1=n_r$ such that $(n_1,n_2,\dots,n_{r})$ are the successive times that verify $(i_1,j_1)=(i_2,j_2)=\dots=(i_r,j_r)$ before a new couple of particles $\{i_{r+1},j_{r+1}\}\neq \{i_1,j_1\}$ is considered, and we let $k_1=r$ be the number of times the couple is repeated. Define then $a_2\leq b_2,a_3\leq b_3,\dots, a_m\leq b_m$ similarly for the next interacting couples, with $m$ denoting the number of alternating couples and 
$k_2,\dots,k_m$ the numbers of times the couples are repeated successively.

Further let $\mathbf{X} = (X_1,\dots,X_m)$ and $\mathbf Y = (Y_1,\dots,Y_m)$
with $X_r = (x_r^1,\dots,x_r^q)$ and $Y_r = {(y_r^1,\dots,y_r^q)}$  denote respectively the positions of the particles at time $a_r$ and $b_r$. We also write  $X=(x_0^p)_{p\leq q}$,
 for the initial positions of the
particles at time $0$.
We call a \emph{diagram} $\mathbf I$ of size $m\in\mathbb N$ any collection of $m$  couples 
$\mathbf I = (i_r,j_r)_{r\leq m}\in\mathcal C_q^m$ such that $\{i_r,j_r\} \neq \{i_{r+1},j_{r+1}\}$. We denote by  $\mathcal D(m,q)$ the set of all diagrams of size $m$.

If we re-write $\Psi_{N,q}(X)$ according to the decomposition that we just described, we find that:
\begin{align*}
&\Psi_{N,q}(X)=\\
&\sum_{m=0}^\infty  \sum_{\substack{1\leq a_1\leq b_1 < a_2 \leq b_2 <  \dots < a_m \leq b_m\leq T\\
\mathbf X, \mathbf Y \in (\mathbb Z^2)^{m\times q}, (i_r,j_r)_{r\leq m} \in \mathcal D(m,q)}} \,\sum_{k_1 \in \llbracket 1,b_1-a_1 + 1 \rrbracket,\dots,k_m \in \llbracket 1, b_m-a_m + 1\rrbracket} \sigma_N^{2k_1+\dots+2k_m}  \\
& \times \prod_{p\leq q} p_{a_1}(x_{1}^p - x_0^p) 
\Big(\prod_{r=1}^{m-1} \prod_{p\leq q} p_{a_{r+1} - b_r}(x_{r+1}^p - y_r^p)\Big)
\prod_{r=1}^{m} \mathbf{1}_{x_r^{i_r} = x_r^{j_r}} \mathbf{1}_{y_r^{i_r} = y_r^{j_r}}     \\
&  \times  \sum_{\substack{a_r<n_1<\dots<n_{k_r-2} < b_r\\ \mathbf z^1 \in (\mathbb Z^2)^{k_r-2},\dots,\mathbf z^q \in (\mathbb Z^2)^{k_r-2}}} \prod_{s=1}^{k_r-2} \mathbf{1}_{z_s^{i_r} = z_s^{j_r}} \prod_{p\leq q}   p_{(a_r,n_1,\dots,n_{k_r-2},b_r),(x_r^{p},z_1^p,\dots,z_{k_r-2}^p,y_r^p)}.
\end{align*}
See Figure
\ref{fig:caseslong} for a pictorial description of the intersections associated with a diagram.

Summing over all the configurations between time $a_r$ and $b_r$
 gives a  contribution of $\sigma_N^2 \mathbf{1}_{x_r^{i_r} = y_r^{i_r}}$ when $a_r = b_r$, and
\begin{align*}
&\sum_{k=2}^\infty \sigma_N^{2k} \sum_{\substack{n_0=a_r < n_1 <\dots < n_{k-2}< b_r = n_{k-1}  \\x_0=x_r^{i_r} ,x_1,\dots,x_{k-2}\in \mathbb Z^2,x_{k-1}=y_r^{i_r}}}  \prod_{i=1}^{k-1} p_{n_{i}-n_{i-1}}(x_i-x_{i-1})^2\\
& = \sigma_N^2 U_N(b_r-a_r,y_r^{i_r}-x_r^{i_r}),
\end{align*}
when $a_r<b_r$ (in this case $k_r \geq 2$ by definition).
It directly follows that:
\begin{equation}\label{eq:decompambm0}\Psi_{N,q}(X) =  \sum_{m=0}^\infty \sigma_N^{2m} \sum_{\substack{1\leq a_1\leq b_1 < a_2 \leq b_2 <  \dots < a_m \leq b_m\leq T\\
\mathbf X, \mathbf Y \in \mathbb Z^{m\times q}, \mathbf I =(i_r,j_r)_{r\leq m} \in \mathcal D(m,q)}} A_{X,\mathbf a, \mathbf b, \mathbf X,\mathbf Y,\mathbf I},
\end{equation}
where
\begin{align}
  \label{eq-Astar}
A_{X,\mathbf a, \mathbf b, \mathbf X,\mathbf Y,\mathbf I} & =  \prod_{p\leq q} p_{a_1}(x_1^p-x_0^p) \prod_{r=1}^m  U_N(b_r-a_r,y_r^{i_r}-x_r^{i_r})  \mathbf{1}_{x_r^{i_r} = x_r^{j_r}} \mathbf{1}_{y_r^{i_r} = y_r^{j_r}}
\nonumber \\
&\times \prod_{p\notin\{i_r,j_r\}}p(b_r-a_r,y_r^p-x_r^p) \prod_{r=1}^{m-1}\prod_{p\leq q} p(a_{r+1} - b_r, x_{r+1}^p - y_r^p).
\end{align}

We can further simplify the expression \eqref{eq:decompambm0}. Let $\mathbf I=(i_r,j_r)_{r\leq m}\in \mathcal D(m,q)$ 
be any diagram. For all $r\leq m$, denote by ${\bar k^1_{r}}$ the last index $l<r$ such that $i_r\in \{i_l,j_l\}$, i.e.\
${\bar k^1_{r}} = \sup\{l\in \llbracket 1,r-1\rrbracket : i_r \in \{i_l,j_l\}\}$. When the set is empty we set ${\bar k_r^1}=0$. Define ${\bar k^2_{r}}$ similarly for $j_r$ instead of $i_r$ and let ${\bar k_r = \bar
k_r^1 \vee \bar k_r^2}$. See figure \ref{fig:caseslong}.
\begin{proposition} For all $X\in (\mathbb Z^2)^q$,
\begin{equation}\label{eq:decompambm}\Psi_{N,q}(X) =  \sum_{m=0}^\infty \sigma_N^{2m} \sum_{\substack{1\leq a_1\leq b_1 < a_2 \leq b_2 <  \dots < a_m \leq b_m\leq T\\
\mathbf x, \mathbf y\in \mathbb Z^{m}, \mathbf I =(i_r,j_r)_{r\leq m} \in \mathcal D(m,q)}} \tilde A_{X,\mathbf a, \mathbf b, \mathbf x,\mathbf y,\mathbf I},
\end{equation}
where
\begin{align}
  \label{eq-Asstar}
  &\tilde A_{X,\mathbf a, \mathbf b, \mathbf x,\mathbf y,\mathbf I}  =  \prod_{p\in \{i_1,j_1\}} p_{a_1}({x_{1}-x_0^p}) \prod_{r=1}^m  U_N(b_r-a_r,y_r-x_r)
\nonumber \\
&\times \prod_{r=1}^{m-1} p(a_{r+1} - b_{{\bar k}_{r+1}^1}, x_{r+1} - y_{
{\bar k}_{r+1}^1}) p(a_{r+1} - b_{{\bar k}_{r+1}^2}, x_{r+1} -
y_{{\bar k}_{r+1}^2}).
\end{align}
\end{proposition}
\begin{proof}
  Denote $x_r = x_r^{i_r}$ and $y_r=y_r^{i_r}$. {We obtain \eqref{eq-Asstar} from \eqref{eq-Astar} by using the semi group property of the random walk transition probabilities and summing, at intersection times,
  over the location of particles not involved in the intersection.}
\end{proof}
\begin{figure}[ht]
  \centering
  \begin{minipage}[b]{1\linewidth}
\begin{tikzpicture}
  \begin{scope}
\coordinate (A) at (-5.6,1);
\draw (A) node[left] {$i_{m-1}$};
\coordinate (B) at (-5.6,0);
\draw (B) node[left] {$j_{m-1}$};
\coordinate (D) at (-5.6,-2);
\draw (D) node[left] {$j_{m}$};
\coordinate (E) at (-4.6,-2.5);
\draw (E) node[below] {$a_{m-1}$};
\draw[.] (-4.6,-2.5)--(-4.6,0.5);
\coordinate (F) at (-2.6,-2.5);
\draw (F) node[below] {$b_{m-1}$};
\draw[.] (-2.6,-2.5)--(-2.6,0.5);
\coordinate (EE) at (-1.6,-2.5);
\draw (EE) node[below] {$a_{m}$};
\draw[.] (-1.6,-2.5)--(-1.6,-1.5);
\coordinate (FF) at (-0,-2.5);
\draw (FF) node[below] {$b_{m}$};
\draw[.] (0,-2.5)--(0,-1.5);

\draw[thick] (-2.6,0.5) .. controls (-1,1.5) .. (2,1);
\draw[thick] (-2.6,0.5) .. controls (-2,-0.5) .. (-1.6,-1.5);
\draw[thick] (-3,0.5) .. controls (-2.8,0.8) .. (-2.6,0.5);
\draw[thick] (-3,0.5) .. controls (-2.8,0.2) .. (-2.6,0.5);
\draw[thick] (-3.4,0.5) .. controls (-3.2,0.8) .. (-3,0.5);
\draw[thick] (-3.4,0.5) .. controls (-3.2,0.2) .. (-3,0.5);
\draw[thick] (-3.8,0.5) .. controls (-3.6,0.8) .. (-3.4,0.5);
\draw[thick] (-3.8,0.5) .. controls (-3.6,0.2) .. (-3.4,0.5);
\draw[thick] (-4.2,0.5) .. controls (-4,0.8) .. (-3.8,0.5);
\draw[thick] (-4.2,0.5) .. controls (-4,0.2) .. (-3.8,0.5);
\draw[thick] (-4.6,0.5) .. controls (-4.4,0.8) .. (-4.2,0.5);
\draw[thick] (-4.6,0.5) .. controls (-4.4,0.2) .. (-4.2,0.5);
\draw[thick] (-5.6,1) .. controls (-5,1.5) .. (-4.6,0.5);
\draw[thick] (-5.6,0) .. controls (-5,-0.5) .. (-4.6,0.5);

\draw[thick] (-0,-1.5) .. controls (1,-2) .. (2,-2);
\draw[thick] (-0,-1.5) .. controls (1,-0.5) .. (2,-0.5);
\draw[thick] (-1.6,-1.5) .. controls (-1.4,-1.2) .. (-1.2,-1.5);
\draw[thick] (-1.6,-1.5) .. controls (-1.4,-1.8) .. (-1.2,-1.5);
\draw[thick] (-1.2,-1.5) .. controls (-1,-1.2) .. (-0.8,-1.5);
\draw[thick] (-1.2,-1.5) .. controls (-1,-1.8) .. (-0.8,-1.5);
\draw[thick] (-0.8,-1.5) .. controls (-0.6,-1.2) .. (-0.4,-1.5);
\draw[thick] (-0.8,-1.5) .. controls (-0.6,-1.8) .. (-0.4,-1.5);
\draw[thick] (-0.4,-1.5) .. controls (-0.2,-1.2) .. (-0,-1.5);
\draw[thick] (-0.4,-1.5) .. controls (-0.2,-1.8) .. (-0,-1.5);
\draw[thick] (-5.6,-2) .. controls (-3,-2.5) .. (-1.6,-1.5);

  \end{scope}

\end{tikzpicture}
\end{minipage}
  \begin{minipage}[b]{1\linewidth}
\begin{tikzpicture}
  \begin{scope}
\coordinate (AA) at (-7.6,1);
\draw (AA) node[above] { $\stackrel{  \;\mbox{\rm exchanges}}{\cdots}$};
\coordinate (AAA) at (-7.5,-2.6);
\draw (AAA) node[above] {$\stackrel{  \;\mbox{\rm exchanges}}{\cdots}$};
\coordinate (A) at (-5.6,1);
\draw (A) node[left] {$i_{m-1}$};
\coordinate (B) at (-5.6,0);
\draw (B) node[left] {$j_{m-1}$};
\coordinate (C) at (-5.6,-1);
\draw (C) node[left] {$i_{m}$};
\coordinate (D) at (-5.6,-2);
\draw (D) node[left] {$j_{m}$};
\coordinate (E) at (-4.6,-2.5);
\draw (E) node[below] {$a_{m-1}$};
\draw[.] (-4.6,-2.5)--(-4.6,0.5);
\coordinate (F) at (-2.6,-2.5);
\draw (F) node[below] {$b_{m-1}$};
\draw[.] (-2.6,-2.5)--(-2.6,0.5);
\coordinate (EE) at (-1.6,-2.5);
\draw (EE) node[below] {$a_{m}$};
\draw[.] (-1.6,-2.5)--(-1.6,-1.5);
\coordinate (FF) at (-0,-2.5);
\draw (FF) node[below] {$b_{m}$};
\draw[.] (0,-2.5)--(0,-1.5);

\coordinate (EEE) at (-8.6,-2.5);
\draw (EEE) node[below] {$a_{\bar k_m}$};
\draw[.] (-8.6,-2.5)--(-8.6,0.5);

\coordinate (FFF) at (-9,0.5);
\draw (FFF) node[left]{$\cdots$};
\coordinate (FFFF) at (-10,-2);
\draw (FFFF) node[left]{$\cdots$};

\draw[thick] (-9,0.5) .. controls (-8.8,0.8) .. (-8.6,0.5);
\draw[thick] (-9,0.5) .. controls (-8.8,0.2) .. (-8.6,0.5);
\draw[thick] (-8.6,0.5) .. controls (-7,-1) .. (-5.6,-1);
\draw[thick] (-8.6,0.5) .. controls (-8.2,1.1) .. (-7.8,1);

\draw[thick] (-10,-2) .. controls (-9.8,-1.7) .. (-9.6,-2);
\draw[thick] (-10,-2) .. controls (-9.8,-2.3) .. (-9.6,-2);
\draw[thick] (-9.6,-2) .. controls (-7,-1.7) .. (-5.6,-2);
\draw[thick] (-9.6,-2) .. controls (-9,-2.3) .. (-8,-2.5);

\draw[thick] (-2.6,0.5) .. controls (-1,1.5) .. (2,1);
\draw[thick] (-2.6,0.5) .. controls (-1,-0.5) .. (2,0);
\draw[thick] (-3,0.5) .. controls (-2.8,0.8) .. (-2.6,0.5);
\draw[thick] (-3,0.5) .. controls (-2.8,0.2) .. (-2.6,0.5);
\draw[thick] (-3.4,0.5) .. controls (-3.2,0.8) .. (-3,0.5);
\draw[thick] (-3.4,0.5) .. controls (-3.2,0.2) .. (-3,0.5);
\draw[thick] (-3.8,0.5) .. controls (-3.6,0.8) .. (-3.4,0.5);
\draw[thick] (-3.8,0.5) .. controls (-3.6,0.2) .. (-3.4,0.5);
\draw[thick] (-4.2,0.5) .. controls (-4,0.8) .. (-3.8,0.5);
\draw[thick] (-4.2,0.5) .. controls (-4,0.2) .. (-3.8,0.5);
\draw[thick] (-4.6,0.5) .. controls (-4.4,0.8) .. (-4.2,0.5);
\draw[thick] (-4.6,0.5) .. controls (-4.4,0.2) .. (-4.2,0.5);
\draw[thick] (-5.6,1) .. controls (-5,1.5) .. (-4.6,0.5);
\draw[thick] (-5.6,0) .. controls (-5,-0.5) .. (-4.6,0.5);

\draw[thick] (-0,-1.5) .. controls (1,-2) .. (2,-2);
\draw[thick] (-0,-1.5) .. controls (1,-0.5) .. (2,-0.5);
\draw[thick] (-1.6,-1.5) .. controls (-1.4,-1.2) .. (-1.2,-1.5);
\draw[thick] (-1.6,-1.5) .. controls (-1.4,-1.8) .. (-1.2,-1.5);
\draw[thick] (-1.2,-1.5) .. controls (-1,-1.2) .. (-0.8,-1.5);
\draw[thick] (-1.2,-1.5) .. controls (-1,-1.8) .. (-0.8,-1.5);
\draw[thick] (-0.8,-1.5) .. controls (-0.6,-1.2) .. (-0.4,-1.5);
\draw[thick] (-0.8,-1.5) .. controls (-0.6,-1.8) .. (-0.4,-1.5);
\draw[thick] (-0.4,-1.5) .. controls (-0.2,-1.2) .. (-0,-1.5);
\draw[thick] (-0.4,-1.5) .. controls (-0.2,-1.8) .. (-0,-1.5);
\draw[thick] (-5.6,-1) .. controls (-3,-0.5) .. (-1.6,-1.5);
\draw[thick] (-5.6,-2) .. controls (-3,-2.5) .. (-1.6,-1.5);

  \end{scope}

\end{tikzpicture}
\end{minipage}
\caption{Two types of diagrams. Note the different types of exchanges. In
the top diagram, $\bar k_m=m-1$ and the $m$th
jump is considered
short (the notion of short and long jumps
is defined in Section \ref{subsec-induction}). In the bottom, the $m$th jump is considered long (with respect to a given $L$) if
$m-\bar k_m>L+2$.
In that case,
both paths $i_m,j_m$ will be involved in an intersection not before
$a_{m-L-2}$.}
\label{fig:caseslong}
\end{figure}
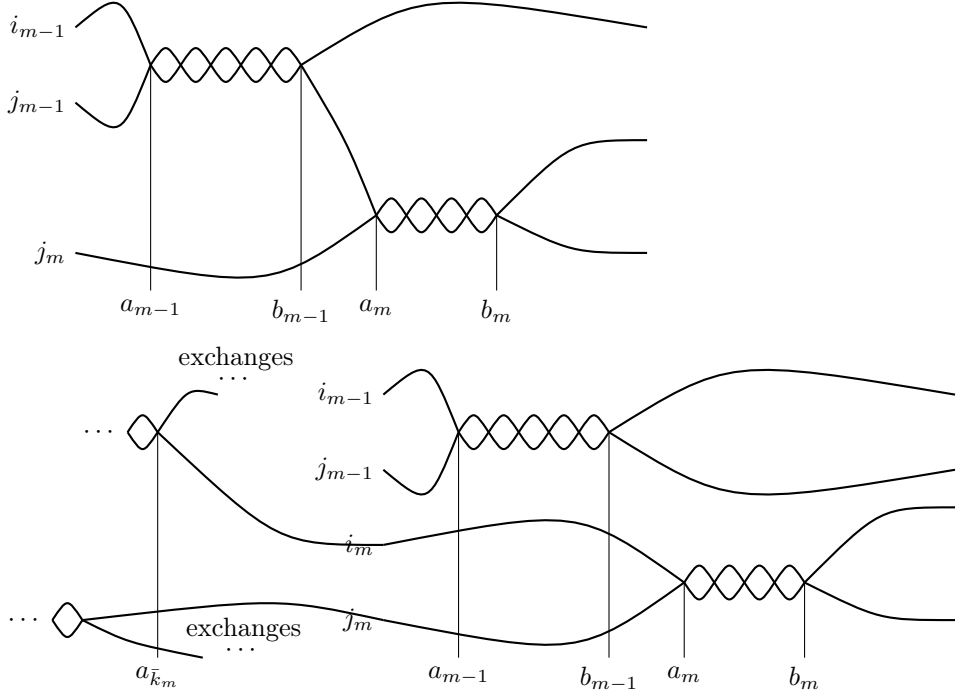

\begin{proposition}
  \label{prop-3.3}
  We have that
\begin{equation} \label{eq:boundMomentAmn}
\sup_{X\in (\mathbb Z^2)^q} \Psi_{N,q}(X) \leq  \sum_{m=0}^\infty \sum_{\mathbf I\in \mathcal D(m,q)} \sigma_N^{2m} A_{m,N,\mathbf I},
\end{equation}
where
\begin{equation} \label{eq:BoundAmN}
\begin{aligned}
&A_{m,N,\mathbf I} =
  \sum_{\substack{u_i \in \llbracket 1,T \rrbracket,v_i\in \llbracket 0,T\rrbracket , 1\leq i \leq m \\ \sum_{i=1}^m {u_i} \leq T}}    p_{2u_1}^\star U_N(v_m)  \prod_{r=1}^{m-1}  U_N(v_{r})
p^\star_{v_r + 2u_{r+1}+ 2\tilde u_{r+1}}.
\end{aligned}
\end{equation}
with
\begin{equation}
  \label{eq-tildeudef}
  \tilde u_{r} = \begin{cases}   \sum_{i={\bar k}_{r}+1}^{r-1} u_{i}
    & \text{if}\quad {\bar k}_r<r-1,\\
  \frac{u_{r-1}}{2} & \text{if}\quad  {\bar k}_r=r-1,
\end{cases}
\end{equation}
and $p^\star_k = \sup_{x\in \mathbb Z^2} p_k(x)$.
\end{proposition}
\begin{proof}
By \eqref{eq:decompambm}, it is enough to show that
\begin{equation} \label{eq:propAmn}
\sup_{X\in (\mathbb Z^2)^q}  \sum_{\substack{1\leq a_1\leq b_1 < a_2 \leq b_2 <
  \dots < a_m \leq b_m\leq T\\
\mathbf x, \mathbf y \in \mathbb Z^{m}}} \tilde A_{X,\mathbf a, \mathbf b, \mathbf x,\mathbf y,\mathbf I} \leq A_{m,N,\mathbf{I}}.
\end{equation}
We begin by summing on $y_m$, which gives a contribution of
\[\sum_{y_m} U_N(b_m-a_m,y_m-x_m) = U_N(b_m-a_m),\]
where $U_N(n)$ is defined in \eqref{eq:def_UN}. Then summing on $x_m$ gives a factor
\begin{align*}
  &\sum_{x_m} p(a_m-b_{{\bar k^1_m}},x_m-y_{{\bar k^1_m}})p(a_m-
  b_{{\bar k^2_m}},x_m-y_{{\bar k^2_m}})\\
  & = p(2a_m-b_{{\bar k^1_m}}-b_{{\bar k^2_m}},y_{{\bar k^1_m}}
  -y_{{\bar k^2_m}}) \leq p_{2a_m-b_{{\bar k^1_m}}-b_{{\bar k^2_m}}}^\star.
\end{align*}
By iterating this process we obtain that the sum on $\mathbf x,\mathbf y$ is bounded (uniformly on the starting point $X$) by
\[p_{2a_1}^\star U_N(b_m-a_m)  \prod_{r=1}^{m-1}  U_N(b_r-a_r)
p^\star_{2a_{r+1}-b_{{\bar k^1_{r+1}}}-b_{{\bar k^2_{r+1}}}}.
\]
If we introduce the change of variables $u_i=a_i-b_{i-1}$ and $v_i=b_i-a_i$ with $b_0=0$, then equation \eqref{eq:propAmn} follows from combining that
$2 a_{r+1}-b_{{\bar k_{r+1}^1}}-b_{{\bar k_{r+1}^2}} \geq v_r + 2u_{r+1}+ 2 \tilde{u}_{r+1}$
with the monotonicity of $p_n^\star$ in $n$,
which follows from
\begin{equation}
  \label{eq-monotone}
  p_{n+1}^\star=\sup_y \sum_x p_n(x) p_1(y-x)\leq
p_n^\star.
\end{equation}
\end{proof}
\subsection{Estimates on $U_N$}
It is clear from Proposition \ref{prop-3.3} that the function $U_N$ plays
a crucial rote in our moment estimates, which we will
obtain by an induction in the next subsection. In the current
subsection, we digress and obtain
a-priori estimates on $U_N$ (and $\IE[W_N(\beta_N)^2]$).
Appendix \ref{app-UN} contains some improvements that
are not needed in the current work but may prove useful in follow up work.
\begin{proposition} There exists $N_0=N_0(\hat \beta)$ such that for all $N\geq N_0$ 
and all $n\leq N$,
\begin{equation} \label{eq:bound2ndMoment}
\IE\left[W_n(\beta_N)^2\right]  \leq \frac{1}{1-\sigma_N^2 R_n}.
\end{equation}
Furthermore, there exists  $\varepsilon_{n}=\varepsilon(n,\hat \beta)\to 0$ as $n\to\infty$, such that
for all  $N\geq n$,
\begin{equation} \label{eq:asympt2ndMoment}
\IE\left[W_n(\beta_N)^2\right] = (1+\varepsilon_n) \frac{1}{1-\hat \beta^2 \frac{\log n}{\log N}}.
\end{equation}
\end{proposition}
\begin{proof} We first choose $N_0=N_0(\hat \beta)$ large enough such that for all $N\geq (n\vee N_0)$, we have $\sigma_N^2 R_n<1$.
  That this is possible follows from 
  \eqref{eq:pnstar} which yields that
 \begin{equation} \label{eq:upper_boundRn}
\forall n\in \mathbb N,\quad R_n = \sum_{s=1}^n p_{2s}(0) \leq \frac{1}{\pi}\sum_{s=1}^n \frac{1}{s} \leq \frac{1}{\pi}\log (n+1).
\end{equation}
For the rest of the proof, we continue in this setup.
Similarly to \eqref{eq:chaosSimple}, we have (letting $n_0=x_0=0$) that
\begin{align}
\IE\left[W_n(\beta_N)^2\right] & = \DE_0\left[e^{\beta_N^2 \sum_{k=1}^n \mathbf{1}_{S_k^1 = S_k^2}} \right]\nonumber\\
& = \sum_{k=0}^\infty \sigma_N^{2k} \sum_{0< n_1 <\dots < n_k \leq n} \sum_{x_1,\dots,x_{k}\in \mathbb Z^2} \prod_{i=1}^{k}  p_{n_{i}-n_{i-1}}(x_i-x_{i-1})^2. \label{eq:whereRunFree}
\end{align}
Hence, if for each $k\in \mathbb N$ in  \eqref{eq:whereRunFree} we let $n_{i}-n_{i-1}$ run free in $\llbracket 1,n\rrbracket$, we obtain that
\begin{align*}
\IE\left[W_n(\beta_N)^2\right] & \leq \sum_{k=0}^\infty \sigma_N^{2k}  \left( \sum_{m=1}^n \sum_{x\in \mathbb Z^2} p_{m}(x)^2\right)^k
= \sum_{k=0}^\infty  \sigma_N^{2k} R_n^k = \frac{1}{1- \sigma_N^2 R_n},
\end{align*}
which gives \eqref{eq:bound2ndMoment}.
On the other hand, if for each $k\in \mathbb N$ in  \eqref{eq:whereRunFree} we let $n_{i}-n_{i-1}$ run free in $\llbracket 1, n/k \rrbracket$, we have
\begin{align*}
\IE\left[W_n(\beta_N)^2\right] & \geq 1+ \sum_{k=1}^\infty \sigma_N^{2k}  \left( \sum_{m=1}^{n/k} \sum_{x\in \mathbb Z^2} p_{m}(x)^2\right)^k\\
& \geq 1+ \sum_{k=1}^{\log n}  \sigma_N^{2k} R_{n/\log n}^k = \frac{1-(\sigma_N^2 R_{n/\log n})^{\log n+1}}{1- \sigma_N^2 R_{n/\log n}},
\end{align*}
By \eqref{eq:LLTsigmaN} and the fact that $R_n\sim \frac{1}{\pi} \log n$
as $n \to \infty$
by \eqref{eq-RNas},
we find that for all $N\geq n$,
\[\IE\left[W_n(\beta_N)^2\right]  \geq (1+\delta_n) \frac{1}{1-\hat \beta^2 \frac{\log n}{\log N}},
\]
with $\delta_n = \delta_n(\hat \beta)\to 0$ as $n\to\infty$.
Combining this with \eqref{eq:bound2ndMoment} entails \eqref{eq:asympt2ndMoment}.
\end{proof}

\begin{proposition} \label{prop:factor2ndMoment}
For all $M\geq 1$, we have:
\begin{equation} \label{eq:firstequality}
\sum_{n= 0}^M U_N(n) =\IE\left[W_{M}^2\right].
\end{equation}
Moreover, there is $C(\hat \beta)>0$ such that, as $N\to\infty$ and for all $n\leq N$,
\begin{equation} \label{eq:boundP2P2ndMomentC}
 U_N(n) \leq C \frac{1}{\left(1-\hat \beta^2 \frac{\log n}{\log N}\right)^2} \frac{1}{n\log N}.
\end{equation}
\end{proposition}
\begin{remark} When $n\to\infty$, one can take the constant $C$ that appears in \eqref{eq:boundP2P2ndMomentC} arbitrarily close to one. See Appendix \ref{app-UN}.
\end{remark}
\begin{proof}
By \eqref{eq:chaos2}, we have, for $n\geq 1$, with $x_0=0$,
\begin{equation} \label{eq:chaosUN}
\begin{aligned}
& U_N(n)\\
& = \sigma_N^2  \sum_{k=1}^\infty \sigma_N^{2(k-1)}\sum_{0< n_1 <\dots < n_{k-1} < n_k := n} \sum_{x_1,\dots,x_{k}\in \mathbb Z^2} \prod_{i=1}^{k}  p_{n_{i}-n_{i-1}}(x_i-x_{i-1})^2\\
& = \sigma_N^2  \sum_{k=1}^\infty \sigma_N^{2(k-1)}\sum_{0< n_1 <\dots < n_{k-1} < n_k := n} \prod_{i=1}^{k}  p_{2n_{i}-2n_{i-1}}(0).
\end{aligned}
\end{equation}
Therefore,
\begin{align*}
\sum_{n=0}^M U_N(n)
& = 1+ \sum_{k=1}^\infty \sigma_N^{2k} \sum_{0< n_1 <\dots < n_{k-1} < n_k \leq  M} \sum_{x_1,\dots,x_{k}\in \mathbb Z^2} \prod_{i=1}^{k}  p_{n_{i}-n_{i-1}}(x_i-x_{i-1})^2 \\
& = \IE[W_{M}^2],
\end{align*}
which yields \eqref{eq:firstequality}.

We now prove \eqref{eq:boundP2P2ndMomentC} 
by expressing $U_N$ as a function of a renewal process.
From \eqref{eq:chaosUN}, one can see that the following representation for $U_N(n)$ holds when $n\geq 1$:
\[U_N(n) =  \sum_{k=1}^\infty (\sigma_N^2 R_N)^k \DP\left(\tau_k^{(N)} = n\right),\]
where the $\tau_k^{(N)}$ are renewal times defined by
\[\tau_k^{(N)} = \sum_{i\leq k} T^{(N)}_i,\]
with $(T_i^{(N)})_i$ being i.i.d.\ random variables with distribution
\[\DP\left(T^{(N)}_i = n\right) = \frac{1}{R_N} p_{2n}(0) \mathbf{1}_{1\leq n\leq N},\quad \text{and} \quad R_N = \sum_{n=1}^N p_{2n}(0). \]
The renewal formulation that is used here is due to \cite{CaravennaFrancesco2019TDsr}. We also refer to [21, Chapter 1] for the connection of polymer models to renewals processes in a general context. By \cite[Proposition 1.5]{CaravennaFrancesco2019TDsr}, there exists $C>0$ such that for all $n\leq N$,
\begin{equation}\label{eq:dickman}
\DP\left(\tau_k^{(N)} = n\right) \leq C k \DP\left(T_1^{(N)} = n\right) \DP\left(T_1^{(N)} \leq n\right)^{k-1}.
\end{equation}
Hence, using that $\sum_{k=1}^\infty  ka^{k-1} = \frac{1}{(1-a)^2}$ for $a<1$,
\begin{align*}
U_N(n)& \leq C \sum_{k=1}^{\infty} (\sigma_N^2 R_N)^k k \DP\left(T_1^{(N)} = n\right) \DP\left(T_1^{(N)} \leq n\right)^{k-1}\\
& = C \frac{ p_{2n}(0) }{R_N} \frac{\sigma_N^2 R_N}{\left(1-\sigma_N^2 R_N \frac{R_n}{R_N} \right)^2},
\end{align*}
which gives \eqref{eq:boundP2P2ndMomentC} by \eqref{eq-RNas}, \eqref{eq:pnstar}
and
\eqref{eq:LLTsigmaN}.
\end{proof}

\subsection{Summing on the $v_i$'s}
In the following, we denote
\begin{equation}
  \label{eq-F}
F(u) = \frac{1}{u}  \frac{1}{1-\hat \beta^2 \frac{\log (u) }{\log N}}.
\end{equation}
By differentiation with respect to $u$ one checks that
$F$ is non-increasing.
\begin{proposition} \label{prop:boundAmNtildeBis} There exist 
 $N_0(\hat \beta)>0$ and $\varepsilon_N=\varepsilon(N,\hat{\beta})
\searrow 0$ as $N\to\infty$,  such that 
for all $N\geq N_0(\hat \beta)$,
\begin{equation} \label{eq:boundPsiAtilde}
\sup_{X\in (\mathbb Z^2)^q} \Psi_{N,q}(X) \leq \sum_{m=0}^\infty \sigma_N^{2m} \sum_{\mathbf I \in \mathcal D(m,q)} \left(\frac{1}{\pi}\right)^{m-1}  \tilde{A}_{m,N,\mathbf I},
\end{equation}
where, {recalling \eqref{eq-tildeudef},}
\begin{equation}\label{eq:boundAmntildeSansv}
\tilde A_{m,N,\mathbf I} = \frac{1}{1-\hat \beta^2} \sum_{ u_i \in \llbracket 1,T \rrbracket, 1\leq i \leq m} (1+\varepsilon_N)^{m} p^\star_{2u_1} \prod_{r=2}^{m}
F(u_r +\tilde u_{r}) \mathbf{1}_{\sum_{i=1}^{r} u_i \leq T}.
\end{equation}
\end{proposition}
\begin{proof}
  By \eqref{eq:BoundAmN}, \eqref{eq:asympt2ndMoment}  and \eqref{eq:firstequality}, summing over $v_m$ in $A_{m,N,\mathbf I}$ gives a factor bounded by
  $\frac{1}{1-\hat \beta^2}(1+o(1))$. We will now estimate the sum over the variable $v_{m-1}$. Let $w=u_m + \tilde u_m$. 
 (Note that by definition $3/2\leq w\leq T\leq N$,
  and that $w$ might be a non-integer multiple of $1/2$.)
Writing  $v=v_{m-1}$,  the sum over $v_{m-1}$ in \eqref{eq:BoundAmN}
gives a factor 
\begin{equation} \label{eq:sumOnvraw}
\sum_{v=0}^T U_N(v) p^\star_{v+2w}
=: S_{\leq w} + S_{> w} ,
\end{equation}
where $S_{\leq w}$ is the sum on the left hand side of \eqref{eq:sumOnvraw} restricted to $v\leq \lfloor w \rfloor$ and $S_{> w}$ is the sum for $v> \lfloor w \rfloor$. 
Using
\eqref{eq:pnstar} and \eqref{eq:boundP2P2ndMomentC},
 there exists a constant $C=C(\hat \beta)>0$ such that
\begin{equation} \label{bound_SbiggerU}
S_{>w}\leq \frac{C}{\log N}\sum_{v=\lfloor w \rfloor +1}^T \frac{1}{v^2} \leq  \frac{1}{\log N}\frac{C}{w}.
\end{equation}
Using \eqref{eq-monotone} and 
\eqref{eq:firstequality},
\begin{equation} \label{eq:initialestimateSleqU}
\begin{aligned}
S_{\leq w} \leq p^\star_{2w} \sum_{v=0}^{\lfloor w \rfloor} U_N(v)=p^\star_{2w} \IE[W_{\lfloor w \rfloor }^2]& \leq p^\star_{2w}  \frac{1}{1- \sigma_N^2 R_{\lfloor w \rfloor}}.
\end{aligned}
\end{equation}
where the second inequality
holds by \eqref{eq:bound2ndMoment} 
for all $N\geq N_0(\hat \beta)$ since $w\leq N$.
Let $\delta_N = \delta(N,\hat \beta)\to 0$ such that \eqref{eq:LLTsigmaN} holds, and let  $N'_0 = N'_0(\hat \beta)>N_0(\hat \beta)$ be such that $\sup_{N\geq N_0'} \sup_{n\leq N}\hat \beta^2 \frac{1+\log n}{\log N}(1+\delta_N) < 1$. By \eqref{eq:pnstar} and \eqref{eq:upper_boundRn}, we obtain that
\begin{equation*}
p^\star_{2w}  \frac{1}{1- \sigma_N^2 R_w}  \leq   \frac{1}{\pi} \frac{1}{w} \frac{1}{1- \hat \beta^2 \frac{1+\log w}{\log N}(1+\delta_N)}.
\end{equation*}
Moreover, as there is $C(\hat \beta)\in (0,\infty)$ such that
\[\sup_{N\geq N_0'}\sup_{n\leq N} \frac{1}{1- \hat \beta^2 \frac{1+\log n}{\log N}(1+\delta_N)} \leq C(\hat \beta),\]
we see that there exists $\varepsilon'_N=\varepsilon'(N,\hat \b)\searrow_{N\to\infty} 0$ such that for all $n\leq N$,
\begin{align*}
 \left|\frac{1}{1- \hat \beta^2 \frac{1+\log n}{\log N}(1+\delta_N)} - \frac{1}{1- \hat \beta^2 \frac{\log n}{\log N}} \right|  &  \leq \frac{\varepsilon_N'}{1- \hat \beta^2 \frac{\log n}{\log N}}.
\end{align*}
Coming back to \eqref{eq:initialestimateSleqU}, we obtain that for all $N\geq N_0'(\hat \beta)$,
\begin{equation} \label{eq:estimateSleqU}
S_{\leq w} \leq \frac{1}{\pi w} \frac{1+\varepsilon_N'}{1- \hat \beta^2 \frac{\log w}{\log N}}.
\end{equation}
We finally obtain from \eqref{eq:estimateSleqU} and \eqref{bound_SbiggerU} that there exists $\varepsilon'_N=\varepsilon'(N,\b){\searrow_{N\to\infty} 0}$
such that the sum in \eqref{eq:sumOnvraw} is smaller than
\[\left(1+\varepsilon'_N\right)\frac{1}{\pi w}  \frac{1}{1- \hat \beta^2 \frac{\log w}{\log N}} = \left(1+\varepsilon'_N\right)\frac{1}{\pi} F\left(u_m+\tilde u_m\right).
\]
Repeating the same observation for $v_{m-2},\dots,v_{1}$ leads to Proposition \ref{prop:boundAmNtildeBis}.
\end{proof}

\subsection{The induction pattern}
\label{subsec-induction}
Our next goal is to sum over $(u_r)_{r\leq m}$ that appear in \eqref{eq:boundAmntildeSansv}. We will sum by induction starting from $r=m$ and going down to $r=1$. To do so, we first need to define the notion of \emph{good} and  \emph{bad} indices $r$. While performing the induction, encountering a \emph{bad} index will add some nuisance term to the estimate. {We will then show that,}
for typical diagrams, the \emph{bad} indices are rare enough so that the nuisance can be neglected.

 Let $L=L_N\in\mathbb N\setminus\{1,2\}$ to be determined later. Given a diagram $\mathbf I \in \mathcal D(m,q)$, we say that
{$r\in \llbracket 1,m\rrbracket$} is a \emph{long jump} if $r-{\bar k_r} > L + 2$, which means that the last times that the two particles $i_r,j_r$ have been involved in an intersection are not too recent. We say that $r$ is a \emph{small jump} if it is not a long jump. (See Figure
\ref{fig:caseslong} for a pictorial description of short (top) and long (bottom)
jumps).
Since small jumps reduce drastically the combinatorial choice on the new couple that intersects, the diagrams that will contribute to the moments will contain mostly long jumps. Let $K=K(\mathbf I)$ denote the number of {small jumps} and $s_1<\dots<s_{K}$ denote the indices of {small jumps}. We also set $s_{K+1} = m+1$. For all $i\leq K+1$ such that $s_i - s_{i-1} > L+1$, 
we mark the following indices $\{s_{i} - kL-1:k\in \mathbb N, s_{i}-kL
-1>s_{i-1}\}$ as \emph{stopping} indices. We then call any {long jump} $r$ a \emph{fresh} index if $r$ is {stopping} or if $r+1$ is a {small jump}. 
Note that any stopping index is a fresh index. If $m$ is a long jump we also mark it as a fresh index.
The idea is that 
if $k$ is a fresh index and $k-1$ is a  long jump, then $k-1,k-2,\dots$ avoid nuisance terms until 
$k-i$ is a 
stopping index or a small jump; we remark that since our induction will be downward
from $m$, these nuisance-avoiding indices occur in the induction \textit{following}
a fresh index. Hence we say that an index $r$ is \emph{good} if it is a long jump that is
not fresh. An index $r$ is \emph{bad} if it is not good.
For any given diagram, one can easily determine the nature of all indices via the following procedure: (i) mark all small jumps; (ii) mark every stopping index; (iii) mark all fresh indices; (iv) all the remaining indices that have not been marked are good indices.

For all $\mathbf I\in \mathcal D(m,q)$, we define for all $r< m$,
\begin{align*}
  &\varphi(r) = \varphi(r,\mathbf I) = \\
  &\begin{cases}
\inf\{r'\in \llbracket r,m\rrbracket, r' \text{ is fresh}\}-L & \text{if $r$ is not a stopping index and $r+1$ is a long jump},\\
r & \text{otherwise}.
\end{cases}
\end{align*}
We also set $\varphi(m)=m$. 
Here are a few immediate observations:
\begin{lemma} 
\begin{enumerate}[label=(\roman*)]
\item \label{lemmata1} If $r$ is good, then $r+1$ is a long jump.
\item \label{lemmata2} If $r\in \llbracket 2,m-1\rrbracket$ is good, then $\varphi(r-1) = \varphi(r)$.
\item \label{lemmata3} If $r\in \llbracket 2,m\rrbracket$ is fresh, then $\varphi(r-1) = r-L$.
\item \label{lemmata4} If $r\in \llbracket 1,m-1\rrbracket$ is such that $r+1$ is a long jump, then $\varphi(r)\leq r$.
\end{enumerate}
\label{lemmata}
\end{lemma}
\begin{remark} \label{rk:varphi}
  Point \ref{lemmata4} ensures that $\varphi(r)\leq r$ for all $r\leq m$. By \ref{lemmata2}, this implies in turn that $\varphi(r)\leq r-1$ when $r$ is good.
\end{remark}
\begin{proof}
Proof of \ref{lemmata1}. Suppose that $r$ is good. It must be that $r<m$ since by definition $m$ is either fresh or a small jump. Now, $r+1$ must be a long jump otherwise $r$ would be fresh.

Proof of \ref{lemmata2}. Let $r\in \llbracket 2,m-1\rrbracket$ be a good index.  We distinguish two cases. First suppose that $r-1$ is not a stopping index. Then $r-1$ cannot be fresh because $r$ is not a small jump.  Therefore $\varphi(r-1) = \inf\{r' > r-1,r'\text{ is fresh}\}-L$. Furthermore, by \ref{lemmata1}, we have that $\varphi(r) = \inf\{r'\geq r,r'\text{ is fresh}\}-L$ and  thus $\varphi(r-1) = \varphi(r)$.  Now assume that $r-1$ is stopping. Then $\varphi(r-1) = r-1$. Moreover, by definition $r,\dots,r+L-1$ are long jumps and either $r+L-1$ is a stopping index or $r+L$ is a small jump. Therefore $r+L-1$ is a fresh index and $r,\dots,r+L-2$ are good, so that $\varphi(r) = (r+L-1)-L = r -1 = \varphi(r-1)$.

Proof of \ref{lemmata3}. Let $r\in \llbracket 2,m-1\rrbracket$ be a fresh index. We first note that $r-1$ cannot be a stopping index. Indeed, if $r$ is a stopping index, then $r-1$ cannot be stopping by definition; if $r$ is not a stopping index, then as $r$ is fresh, $r+1$ must be a small jump and thus $r-1$ cannot be stopping. Now, as $r-1$ is not stopping and $r$ is fresh, we obtain that $\varphi(r-1) = r-L$. (Note that $r-1$ cannot be fresh because $r-1$ is not stopping and $r$ is a long jump.)

Proof of \ref{lemmata4}. It is enough to show that $r_0 := \inf\{r'\in \llbracket r,m\rrbracket, r' \text{ is fresh}\}\leq r+L$. Since $s_1 = 1$ and $s_{K+1}=m+1$, we can find $i\leq K$ such that $r\in [s_i,s_{i+1})$. 
Now first suppose that $s_{i+1}-r \leq L+ 1$. As $r+1$ is a long jump, $s_{i+1}> r + 1 \geq s_i + 1$ and so $s_{i+1}-1$ is a long jump because it is in $(s_i,s_{i+1})$. Hence $s_{i+1} - 1$ is fresh (note that this remains true when $s_{i+1} = m+1$) and we obtain that $r_0\leq s_{i+1}-1\leq r+L$.  
Otherwise if $s_{i+1}-r > L+ 1$, we can let $r_\star = s_{i+1}-k_0L-1$, $k_0\in \mathbb N$, be the stopping index of $(s_i,s_{i+1})$ that satisfies $r_\star - L < r \leq r_\star$ (this is the smallest stopping time larger than $r$). By definition $r_\star$ is fresh, therefore $r_0 \leq r_\star \leq r+L$.
\end{proof}

For all
$v\in [1,T]$, we further let
\[f(v)= \frac{\log N}{\hat \beta^2}  \log \left(\frac{1-\hat \beta^2\frac{\log v}{\log N}}{1-\hat \beta^2\frac{\log T}{\log N}}\right).\]
Note that $f$ is non-increasing.
Recall \eqref{eq-tildeudef} and the definition of $F$ in \eqref{eq-F}.
\begin{lemma} \label{lem:fibo}
 For all
  $m\geq 2$, $\mathbf I\in \mathcal D(m,q)$,  $k\in \llbracket 1,m-1 \rrbracket$ and $\sum_{i=1}^{m-k} u_i \leq T$ with $u_i\in \llbracket 1,T\rrbracket$,
\begin{equation} \label{eq:lemmaInduction}
\begin{aligned}
& \sum_{u_i \in \llbracket 1,T\rrbracket, m-k < i \leq m }  \prod_{r=m-k+1}^{m}
F(u_r + \tilde u_r) \mathbf{1}_{\sum_{i=1}^{m-k+1} u_i \leq T}  \\
&\leq \sum_{i=0}^k \frac{c_{i}^k}{(k-i)!}  \frac{1}{\left(1-\hat \beta^2 \right)^{i}} f\left(\sum_{i=\varphi(m-k)}^{m-k} u_i\right)^{k-i}.
\end{aligned}
\end{equation}
with $c_0^1=1$, $c_1^1=2$, $c_i^{k+1}= c_{i}^k +2\gamma_{k}^m 
\sum_{j=0}^{i-1} c_{j}^k$ for $i\leq k+1$ with $\gamma_k^m = \mathbf{1}_{m-k \text{ is bad}}$ and $c_i^k=0$ for $i>k$.
\end{lemma}
\begin{remark}
The $c_i^k$'s depend on $m$ and $\mathbf I\in \mathcal D(m,q)$.
\end{remark}
Before turning to the proof, we need another result that plays a key role in the proof of Lemma \ref{lem:fibo} and which clarifies the role of good indices.
\begin{lemma} \label{lem:indReduction}
  For all $k\in \llbracket 0,m-2 \rrbracket$, $j\leq k$ and $\sum_{i=1}^{m-k-1} u_i \leq T$ with $u_i\in \llbracket 1,T\rrbracket$,
\begin{equation}\label{eq:indReduction}
\begin{aligned}
&S_kf^j(u_1,\dots,u_{m-k-1}) := \\
&  \sum_{u_{m-k}=1}^T
F(u_{m-k} + \tilde u_{m-k})  f\left(\sum_{i=\varphi(m-k)}^{m-k} u_i\right)^{j} \mathbf{1}_{\sum_{i=1}^{m-k} u_i \leq T}
\\ & \leq  \frac{1}{j+1} f\left(\sum_{i=\varphi(m-k-1)}^{m-k-1} u_i\right)^{j+1} \\
&+ \gamma_k^m  \sum_{l=1}^{j+1} \frac{j!}{(j+1-l)!} \frac{2}{\left(1-\hat \beta^2\right)^{l}} f\left(\sum_{i=\varphi(m-k-1)}^{m-k-1} u_i\right)^{j+1-l}.
\end{aligned}
\end{equation}
\end{lemma}
\begin{remark}
  When $m-k$ is good, the {right-hand side of \eqref{eq:indReduction}}
  is reduced to a single term. When $m-k$ is bad, a nuisance term appears.
\end{remark}
\begin{proof}
  We divide the proof {into} three cases.

  \textbf{Case 1:} $m-k$ is good. Necessarily $m-k+1$ is a long jump
  by Lemma \ref{lemmata}-\ref{lemmata1}, 
  so if we let $r_{\text{fresh}} = \inf\{r'>m-k,r' \text{ is fresh}\}$, then $\varphi(m-k)=
  r_{\text{fresh}} - L$. By Remark \ref{rk:varphi} and since $r_{\text{fresh}} > m-k$, we have
\[(m-k)-L < \varphi(m-k) \leq (m-k)-1.\]
Define
\[v:=\sum_{i=\varphi(m-k)}^{m-k-1} u_i \in \llbracket 1,T\rrbracket .\]
 As $m-k$ is a long jump, we first observe that
\begin{equation} \label{eq:umtildebig}
\tilde u_{m-k} \geq u_{m-k-1} +\dots + u_{m-k-L-1}\geq v.
\end{equation}
Since $F$ and $f$ are non-increasing, see \eqref{eq-F},
this implies that
\begin{align*}
&S_kf^j \leq \sum_{u_{m-k}=1}^T  F\left(u_{m-k}+v\right) f\left(u_{m-k} + v\right)^{j} \mathbf{1}_{u_{m-k} + v \leq T} \\ &\leq \int_{v}^T  \frac{1}{u} \frac{1}{1-\hat \beta^2 \frac{\log (u) }{\log N}} f(u)^{j} \dd u
 = \left[ - \frac{1}{j+1} f(x)^{j+1} \right]^{T}_{v} =  \frac{1}{j+1} f(v)^{j+1},
\end{align*}
where in the comparison to the integral, we have used that
$F(x) f(x)^j$ decreases in $x\in [1,\ldots T]$.
Given that $\varphi(m-k-1) = \varphi(m-k)$ by Lemma \ref{lemmata}-\ref{lemmata2}, we have the identity $v = \sum_{i=\varphi(m-k-1)}^{m-k-1} u_i$. Hence \eqref{eq:indReduction} holds.

\textbf{Case 2:} $m-k$ is fresh. By Lemma \ref{lemmata}-\ref{lemmata3}, we have $\varphi(m-k-1) = m-k-L$. Note that in contrast with Case 1, we have from the definition that $\varphi(m-k)=m-k$ and therefore
we cannot follow the same argument as for Case 1. Instead, we now
define $v$ summing from $\varphi(m-k-1)$ and not $\varphi(m-k)$, i.e.
\[v:=\sum_{i=\varphi(m-k-1)}^{m-k-1} u_i \in \llbracket 1,T\rrbracket.\] 
We then decompose
\begin{equation} \label{eq:decSk} S_k f^j = S_k^{\leq v} f^j + S_k^{> v} f^j,
\end{equation} where $S_k^{\leq v} f^j$ is the restriction of the sum in $S_k f^j$ to $u_{m-k} \in \llbracket 1,v\rrbracket$. 
Given that $m-k$ is a long jump, the bounds \eqref{eq:umtildebig} hold again.
Hence, using that $F$ and $f$ are non-increasing,
we find that
\begin{equation} \label{eq:sameWithv}
\begin{aligned}
S_k^{\leq v} f^j & \leq \sum_{u_{m-k}=1}^{v} \frac{1}{u_{m-k} + v} \frac{1}{1-\hat \beta^2 \frac{\log (u_{m-k}+v) }{\log N}} f(u_{m-k})^j \mathbf{1}_{u_{m-k} + v\leq T}\\
&\leq \frac{1}{v} \frac{1}{1-\hat \beta^2} \sum_{u_{m-k}=1}^{v}  f(u_{m-k})^j\\
& \leq \frac{1}{v} \frac{1}{1-\hat \beta^2}  \left(f(1)^j
+ \int_{1}^v f(x)^j\dd x\right),
\end{aligned}
\end{equation}
by comparison to a integral (using that $f$ is non-increasing).
By integrating by part
and using that $f'(x)= -\frac{1}{x}\frac{1}{1-\hat \beta^2 \frac{\log (x)}{\log N}}$, we see that for all $j\geq 1$,
\begin{align*}
f(1)^j + \int_{1}^v f(x)^j\dd x &= vf(v)^j - j\int_{1}^v xf'(x) f(x)^{j-1}\dd x  \\
&\leq vf(v)^j + \frac{j}{1-\hat \beta^2} \int_{1}^v f(x)^{j-1}\dd x.
\end{align*}
If we iterate the integration by parts, we obtain that
\begin{align*}
 f(1)^j + \int_{1}^v f(x)^j\dd x \leq v \sum_{i=0}^j \frac{j!}{(j-i)!}\left(\frac{1}{1-\hat \beta^2 }\right)^i  f(v)^{j-i},
\end{align*}
and so
\begin{equation} \label{eq:IPP}
S_k^{\leq v} f^j \leq \sum_{i=0}^j \frac{j!}{(j-i)!}\left(\frac{1}{1-\hat \beta^2 }\right)^{i+1}   f(v)^{j-i}.
\end{equation}
On the other hand, 
we have
\begin{align*}
S_k^{> v} f^j& \leq \sum_{u_{m-k}=v+1}^T \frac{1}{u_{m-k}} \frac{1}{1-\hat \beta^2 \frac{\log (u_{m-k}) }{\log N}} f(u_{m-k})^j\\
&\leq \int_{v}^T \frac{1}{x} \frac{1}{1-\hat \beta^2 \frac{\log (x) }{\log N}} f(x)^j \dd x
 = \left[ - \frac{1}{j+1} f(x)^{j+1} \right]^{T}_v
 \leq  \frac{1}{j+1} f(v)^{j+1}.
\end{align*}
Combining the two previous estimates yields \eqref{eq:indReduction}.

\textbf{Case 3:} $m-k$ is a small jump. We have that $f(\sum_{i=\varphi(m-k)}^{m-k} u_i)\leq f(u_{m-k})$ as $f$ is non-increasing. Moreover, $\tilde u_{m-k} \geq \frac{u_{m-k-1}}{2}$ always holds. Hence, if we use the same decomposition as in \eqref{eq:decSk} with $v=u_{m-k-1}$, we find using that $F$ is non-increasing that
\begin{align*}
S_k^{\leq v} {f^j} & \leq \sum_{u_{m-k}=1}^{v} \frac{1}{u_{m-k} + v/2} \frac{1}{1-\hat \beta^2 \frac{\log (u_{m-k}+v/2) }{\log N}} f(u_{m-k})^j \mathbf{1}_{u_{m-k} + v/2\leq T}\\
& \leq \frac{2}{v} \frac{1}{1-\hat \beta^2}  \left({f(1)^j} + \int_{{1}}^v f(x)^j\dd x\right)\\
& \leq 2\sum_{i=0}^j \frac{j!}{(j-i)!}\left(\frac{1}{1-\hat \beta^2 }\right)^{i+1}   f(v)^{j-i},
\end{align*}
where we have used the integration by parts from Case 2 and that $f$ is non-increasing in the comparison to the integral.
Furthermore, we have $S_k^{> v} \leq \frac{1}{j+1} f(v)^{j+1}$ as in Case 2. Finally, since $m-k$ is not a long jump we have
$\varphi(m-k-1) = m-k-1$ and  therefore \eqref{eq:indReduction} follows.
\end{proof}


\begin{proof}[Proof of Lemma \ref{lem:fibo}]
We prove the lemma by induction on $k$. The case $k=1$ follows from Lemma \ref{lem:indReduction} with $j=k=0$.

Assume now that \eqref{eq:lemmaInduction} holds for some $k\in \llbracket 1,m-2 \rrbracket$.
Then by \eqref{eq:indReduction} we obtain that the left hand side of \eqref{eq:lemmaInduction} for the index $k+1$ is smaller than the sum of all the entries of the following matrix, where we have set $\mu = 1-\hat \beta^2$  and $f=f(v)$ with $v= \sum_{i=\varphi(m-k-1)}^{m-k-1} u_i$:
\begin{equation} \label{eq:matrix}
\begin{pmatrix}
\frac{c_0^k}{(k+1)!}f^{k+1}  & \frac{2 \gamma_k^m c_0^k}{k!\mu} f^{k}  & \frac{2\gamma_k^mc_0^k }{(k-1)!\mu^2} f^{k-1}  & \cdots & \frac{2\gamma_k^mc_0^k}{\mu^{k}} f & \frac{2\gamma_k^mc_0^k}{\mu^{k+1}}\\
0 & \frac{c_1^k}{k! \mu} f^{k}  & \frac{2\gamma_k^m c_1^k}{(k-1)!\mu^2} f^{k-1}  & \dots & \frac{2\gamma_k^mc_1^k}{\mu^{k}} f  & \frac{2\gamma_k^mc_1^k}{\mu^{k+1}}\\
0&0 & \frac{c_2^k}{(k-1)!\mu^2} f^{k-1}  & \dots & \frac{2\gamma_k^mc_2^k}{\mu^{k}} f & \frac{2\gamma_k^mc_2^k }{\mu^{k+1}}\\
\vdots& \ddots & &\\
0& 0& 0&\dots & \frac{c_k^k}{\mu^{k}} f & \frac{2\gamma_k^mc_k^k }{\mu^{k+1}}
\end{pmatrix},
\end{equation}
and collecting the coefficients of the various terms $f^r$, $r\in \llbracket 0,k+1\rrbracket$, which amounts to summing over the columns of \eqref{eq:matrix}, gives \eqref{eq:lemmaInduction} for $k+1$. 
\end{proof}

Recall \eqref{eq:boundAmntildeSansv}. Lemma \ref{lem:fibo} yields the following.
\begin{proposition} \label{prop:Amn} There  exists $C=C(\hat \beta)>0$ and $\varepsilon_N=\varepsilon(N,\hat \beta)\to 0$  as $N\to\infty$, such that
\begin{equation} \label{eq:upperBoundAmN2}
\begin{aligned}
&\tilde A_{m,N,\mathbf I} \leq  C (1+|\varepsilon_N|)^{m}    \sum_{i=0}^{m} \frac{ c_{i}^{m}}{(m-i)!}
 \times \left(\frac{\log N}{\hat \beta^2} \right)^{m-i} \frac{1}{\left(1-\hat \beta^2\right)^{i}} \lambda_{T,N}^{2(m-i)},
\end{aligned}
\end{equation}
where $\lambda_{T,N}$ is defined in \eqref{eq:lambdaT}. 
\end{proposition}
\begin{proof}[Proof of Proposition \ref{prop:Amn}]
By Proposition \ref{prop:boundAmNtildeBis} and Lemma \ref{lem:fibo} applied to $k=m-1$, we have:
\begin{align*}
\tilde A_{m,N} & \leq \frac{1}{1-\hat \beta^2} (1+\varepsilon_N)^{m}   \sum_{u_1=1}^T \frac{C}{u_1} \sum_{i=0}^{m-1} \frac{ c_i^{m-1}}{(m-1-i)!} 
 f(u_1)^{m-1-i} \frac{1}{\left(1-\hat \beta^2\right)^{i}} \\
& \leq \frac{C}{1-\hat \beta^2} (1+\varepsilon_N')^{m}  \sum_{i=0}^{m} \frac{ c_i^{m-1}+c_{i-1}^{m-1}\mathbf{1}_{i\geq 1}}{(m-i)!}  
 f(1)^{m-i} 
 \frac{1}{\left(1-\hat \beta^2\right)^{i}} ,
\end{align*}
where in the second inequality, we have used that
\begin{align*}\sum_{u_1=1}^T \frac{1}{u_1} f(u_1)^{m-i}
\leq \sum_{u_1=1}^T  F(u_1) f(u_1)^{m-i}&\leq f(1)^{m-i} + \int_{1}^T  F(u_1) f(u_1)^{m-i}\dd u_1 \\
&= f(1)^{m-i} + \frac{f(1)^{m-i+1}}{m-i+1},
\end{align*}
using that $F(u)f(u)$ is non-increasing in the comparison to the integral.
 This yields  \eqref{eq:upperBoundAmN2} since 
 \[c_i^{m-1}+c_{i-1}^{m-1}\mathbf{1}_{i\geq 1}\leq c_{i}^{m-1} +2\gamma_{m-1}^m 
 \sum_{j=0}^{i-1} c_{j}^{m-1} = c_{i}^m,\] as $\gamma_{m-1}^m=1$ since the index $1$ is always bad (it is a small jump).
\end{proof}

\begin{lemma} \label{lem:ci} For all $\mathbf I\in \mathcal D(m,q)$, for all $k\leq m$:
\begin{equation} \label{eq:inductioncik}
\forall i\leq k, \quad  c_i^k \leq 3^i \prod_{r=1}^{k-1} (1+\gamma_r^m).
\end{equation}
\end{lemma}
\begin{proof}
We prove it by induction on $k$. The estimate holds for $k=1$ since $c_0^1=1$ and $c_1^1=2$. 
Suppose that \eqref{eq:inductioncik} holds for some $k\leq m-1$. Then, for all $i\leq k+1$,
\[c_{i}^{k+1} = c_{i}^k +2 \gamma_{k}^m \sum_{j=0}^{i-1} c_{j}^k \leq \prod_{r=1}^{k-1} (1+\gamma_r^m) \left(3^i  + 2 \gamma_{k}^m  \sum_{j=0}^{i-1} 3^j\right) \leq 3^{i} \prod_{r=1}^{k} (1+\gamma_r^m).\]
\end{proof}

\subsection{Proof of Theorem \ref{th:momentwithoutTriple}}
By Proposition \ref{prop:decInChaos}, it is enough to show that
\begin{equation} \label{eq:enoughTo}
  \sup_{X\in (\mathbb Z^2)^q} \Psi_{N,q}(X) \leq  e^{\lambda_{T,N}^2 \binom{q}{2} + cq^{3/2}+ o(1)q^2},
\end{equation}
for some $c=c(\hat \beta)$.
Using Proposition \ref{prop:boundAmNtildeBis}, we have
\begin{equation} \label{eq:preFinalBoundPsi}
  \sup_{X\in (\mathbb Z^2)^q} \Psi_{N,q}(X) \leq \sum_{m=0}^\infty \sigma_N^{2m} \left(\frac{1}{\pi}\right)^{m-1} \sum_{\mathbf I\in \mathcal D(m,q)} \tilde{A}_{m,N,\mathbf I},
\end{equation}
where \eqref{eq:upperBoundAmN2} gives an upper bound on the $\tilde A_{m,N,\mathbf I}$. Observe that by \eqref{eq:inductioncik}, we have
\[c_{i}^{m} \leq 3^i 2^{\sum_{i=1}^{m-1} \mathbf{1}_{i \text{ is bad}}} 
\leq 3^i 2^{\sum_{i=1}^{m} \mathbf{1}_{i \text{ is bad}}}\leq 3^i 2^{2n(\mathbf I) + m/L+1},\]
where $n(\mathbf I)$ is the number of small jumps in $\mathbf I$.
Indeed, an index $r$ is bad if it is a small jump or a fresh index. 
The number of small jumps is $n(\mathbf I)$.
A fresh index is either equal to $m$, or a stopping index, or an index adjacent to a small jump, so the number of fresh indices is at most $1+n(\mathbf I)$ plus the number of stopping indices. Since stopping indices are spaced at least $L$ steps apart, 
there are at most $m/L$ stopping indices. 
Hence there are at most $2n(\mathbf I) + m/L+1$
bad indices. For a fixed $n\leq m$, let us compute the number of diagrams in $\mathcal D(m,q)$ such that $n(I)=n$. One has first to choose the location of the 
small jumps, which gives $\binom{m}{n}$ possibilities. Now if $m$ is a small jump $(m-{\bar k_m}\leq L+2)$, it means that at least one of the two particles $\{i_m,j_m\}$ is the same as one of the particles $\{i_{m-L-2},j_{m-L-2},\dots,i_{m-1},j_{m-1}\}$, therefore there are at most $2(L+2)q$ choices for the couple $(i_m,j_m)$. On the other hand, if $\{i_m,j_m\}$ is a long jump, there are at most $\binom{q}{2}$ possibilities. By repeating the argument, we finally find that the number of diagrams in $\mathcal D(m,q)$ such that $n(I)=n$ is less than $\binom{m}{n} (2(L+2)q)^n \binom{q}{2}^{m-n}$.

Hence, by \eqref{eq:preFinalBoundPsi}, Proposition \ref{prop:Amn}, Lemma \ref{lem:ci} and \eqref{eq:LLTsigmaN},
there exists $\varepsilon_N \searrow 0$ such that
\begin{align*}
&\sup_{X\in (\mathbb Z^2)^q} \Psi_{N,q}(X) \leq C(\hat \beta)\times\\
& \sum_{m=0}^\infty  (1+\varepsilon_N)^{m} \sum_{n=0}^m \binom{m}{n} (2(L+2)q)^n \binom{q}{2}^{m-n}  \sum_{i=0}^{m} \frac{3^i 2^{2n + m/L+1} }{(m-i)!} \left(\frac{\hat \beta^2}{\log N} \right)^{i}  \frac{\lambda_{T,N}^{2(m-i)}}{\left(1-\hat \beta^2\right)^{i}}.
\end{align*}
The sum over $n$ gives a factor of $(8(L+2)q + \binom{q}{2})^m$. Exchanging the sum in $i$ and $m$ entails
\begin{align*}
\sup_{X\in (\mathbb Z^2)^q} \Psi_{N,q}(X) &\leq C\sum_{i=0}^{\infty} (1+\varepsilon_N)^{i} 3^i 2^{i/L} \left(\frac{\hat \beta^2}{\log N} \right)^{i}   \left(8(L+2)q + \binom{q}{2}\right)^i \frac{1}{\left(1-\hat \beta^2\right)^{i}} \times  \\
& \hspace{10mm} \sum_{m=i}^\infty  (1+\varepsilon_N)^{m-i}\times \left(8(L+2)q + \binom{q}{2}\right)^{m-i} \frac{2^{ (m-i)/L} }{(m-i)!}  \lambda_{T,N}^{2(m-i)}.
\end{align*}
So if we assume that
\begin{equation} \label{eq:conditionOnr}
r = 3(1+\varepsilon_N) \frac{2^{1/L}}{\left(1-\hat \beta^2\right)}  \left(\frac{\hat \beta^2}{\log N} \right)   \left(8(L+2)q + \binom{q}{2}\right) < 1,
\end{equation}
we obtain the bound:
\begin{equation}\label{eq:finalBound}
\sup_{X\in (\mathbb Z^2)^q} \Psi_{N,q}(X) \leq \frac{C}{1-r} e^{(1+\varepsilon_N)(8(L+2)q + \binom{q}{2})2^{1/L}\lambda_{T,N}^2}.
\end{equation}
Taking $L=\lceil \sqrt{q} \rceil$ then gives, together with  \eqref{eq:condition20nr},
the desired conclusion \eqref{eq:enoughTo}, since $2^{1/L}\leq 1+2/L$ for 
$L\geq 1$.
\qed

\section{Discussion and concluding remarks}
\label{sec-discussion}
We collect in this section several comments concerning the results of this paper.

\begin{enumerate}
  \item
    Our results allow one already to obtain some estimates
  on the maximum of $Y_N(x):= \log W_N(\b_N,x)$ over subsets $D\subset [0,1]^2$.
  Specifically, let $\gamma>0$ be given and define
  $Y_N^*=\sup_{x\in D}
  Y_N(x)$.
By  Chebyshev's inequality we have that
\begin{align*}&\IP\left( Y_N^* \geq \delta \sqrt {\log N} \right)  \leq 2 N^{2\gamma} \IP\left(Y_N(0) \geq \delta \sqrt {\log N} \right)\\
& \leq 2 N^{2\gamma}  \IE[W_N^q] e^{-q \delta\sqrt{\log N}}
\leq N^{2\gamma +\frac{q^2\lambda^2}{{2 \log N}}-\frac{q\delta}{\sqrt{\log N}}+o(1)},
\end{align*}
where we used \eqref{eq:estimate} in the last inequality. The optimal $q$ (disregarding the constraint in \eqref{eq:condition20nr})
is $q/\sqrt{\log N}=\delta/\lambda^2 $,
and for that value the right side of the last display decays to $0$
if $\delta^2>4\gamma \lambda^2$. The condition on $q$ in \eqref{eq:condition20nr}
then gives the constraint that
$\gamma< \frac{1}{6} \lambda^2 \frac{1-\hat \beta^2}{\hat \beta^2}$, which for $\hat \beta$ small reduces to $\gamma < 1/6$.
We however do not expect
that this Chebyshev based bound is tight, 
even if one uses the optimal $q$ disregarding
\eqref{eq:condition20nr}; this is due to inherent inhomogeneity in time of
the model. We hope to return to this point in future work.
\item In view of the last sentence in Remark
  \ref{rem-1.3}, it is of interest to obtain a lower bound
  on $\IE [W_N^q]$ that matches
  the upper bound, that is, $\IE[W_N^q]\geq e^{\binom 2 q \lambda^2(1-\varepsilon_N)}$. This can be found in \cite{CZLB}.

\end{enumerate}
\bibliographystyle{plain}


\begin{thebibliography}{}

\end{thebibliography}


\begin{thebibliography}{10}

\bibitem{BRZ}
David Belius, Jay Rosen, and Ofer Zeitouni.
\newblock Tightness for the cover time of the two dimensional sphere.
\newblock {\em Prob. Th. Rel. Fields}, 176:1357--1437, 2020.

\bibitem{Bertini98}
Lorenzo Bertini and Nicoletta Cancrini.
\newblock The two-dimensional stochastic heat equation: renormalizing a
  multiplicative noise.
\newblock {\em Journal of Physics A: Mathematical and General}, 31(2):615--622,
  Jan 1998.

\bibitem{Biskup}
Marek Biskup.
\newblock Extrema of the two-dimensional discrete {G}aussian free field.
\newblock In {\em Random graphs, phase transitions, and the {G}aussian free
  field}, volume 304 of {\em Springer Proc. Math. Stat.}, pages 163--407.
  Springer, Cham, 2020.

\bibitem{BDZ}
Maury Bramson, Jian Ding, and Ofer Zeitouni.
\newblock Convergence in law of the maximum of the two-dimensional discrete
  {G}aussian free field.
\newblock {\em Comm. Pure Appl. Math.}, 69(1):62--123, 2016.

\bibitem{CaraSuZy-universalityrelev}
Francesco Caravenna, Rongfeng Sun, and Nikos Zygouras.
\newblock Universality in marginally relevant disordered systems.
\newblock {\em Ann. Appl. Probab.}, 27(5):3050--3112, 2017.

\bibitem{CaravennaFrancesco2019TDsr}
Francesco Caravenna, Rongfeng Sun, and Nikos Zygouras.
\newblock The {Dickman} subordinator, renewal theorems, and disordered systems.
\newblock {\em Electronic Journal of Probability}, 24, 2019.

\bibitem{CaSuZyCrit18}
Francesco Caravenna, Rongfeng Sun, and Nikos Zygouras.
\newblock On the moments of the (2+1)-dimensional directed polymer and
  stochastic heat equation in the critical window.
\newblock {\em Comm. Math. Phys}, 372:385--440, 2019.

\bibitem{CaSuZy18}
Francesco Caravenna, Rongfeng Sun, and Nikos Zygouras.
\newblock The two-dimensional {KPZ} equation in the entire subcritical regime.
\newblock {\em Annals Probab.}, 48:1086--1127, 2020.

\bibitem{CaSuZyCrit21}
Francesco Caravenna, Rongfeng Sun, and Nikos Zygouras.
\newblock The critical $2d$ stochastic heat flow.
\newblock {\em Inventiones Math.}, 233:325--460, 2023.

\bibitem{ChDu18}
Sourav Chatterjee and Alexander Dunlap.
\newblock Constructing a solution of the $(2+1)$-dimensional {KPZ} equation.
\newblock {\em Annals Probab.}, 48:1014--1055, 2020.

\bibitem{GRV}
Laurent Chevillard, Christophe Garban, R\'{e}mi Rhodes, and Vincent Vargas.
\newblock On a skewed and multifractal unidimensional random field, as a
  probabilistic representation of {K}olmogorov's views on turbulence.
\newblock {\em Ann. Henri Poincar\'{e}}, 20(11):3693--3741, 2019.

\bibitem{CMN}
Reda Chhaibi, Thomas Madaule, and Joseph Najnudel.
\newblock On the maximum of the {${\rm C}\beta {\rm E}$} field.
\newblock {\em Duke Math. J.}, 167(12):2243--2345, 2018.

\bibitem{CN}
Reda Chhaibi and Joseph Najnudel.
\newblock On the circle, {$GMC^\gamma=\underleftarrow{\lim} C\beta E_n$ for
  $\gamma\sqrt{2/\beta}$ ($\gamma\leq 1$)}.
\newblock {\em arXiv:1904.00578}, 2019.

\bibitem{CFLW}
T.~Claeys, B.~Fahs, G.~Lambert, and C.~Webb.
\newblock How much can the eigenvalues of a random {H}ermitian matrix
  fluctuate?
\newblock {\em Duke Math. J.}, 170(9):2085--2235, 2021.

\bibitem{CStFlour}
Francis Comets.
\newblock {\em {Directed polymers in random environments. \'Ecole d'\'Et\'e de
  Probabilit\'es de Saint-Flour XLVI -- 2016.}}
\newblock Cham: Springer, 2017.

\bibitem{CZLB}
Cl\'{e}ment Cosco and Ofer Zeitouni.
\newblock Moments of partition functions of {2D Gaussian} polymers in the weak
  disorder regime - {II}.
\newblock {\em arXiv:2305.05758}, 2023.

\bibitem{DPRZ}
Amir Dembo, Yuval Peres, Jay Rosen, and Ofer Zeitouni.
\newblock Cover times for {Brownian} motion and random walks in two dimensions.
\newblock {\em Ann. Math.}, 160:433--464, 2004.

\bibitem{DRSV}
Bertrand Duplantier, R\'{e}mi Rhodes, Scott Sheffield, and Vincent Vargas.
\newblock Log-correlated {G}aussian fields: an overview.
\newblock In {\em Geometry, analysis and probability}, volume 310 of {\em
  Progr. Math.}, pages 191--216. Birkh\"{a}user/Springer, Cham, 2017.

\bibitem{DuSh}
Bertrand Duplantier and Scott Sheffield.
\newblock Liouville quantum gravity and {KPZ}.
\newblock {\em Invent. Math.}, 185(2):333--393, 2011.

\bibitem{Durrett}
Rick Durrett.
\newblock {\em Probability: theory and examples}.
\newblock Brooks/Cole, third edition, 2004.

\bibitem{ErdosTaylor}
P.~Erdős and S.~J. Taylor.
\newblock Some problems concerning the structure of random walk paths.
\newblock {\em Acta Mathematica Academiae Scientiarum Hungaricae},
  11(1-2):137--162, 1963.

\bibitem{Gu18KPZ2D}
Yu~Gu.
\newblock Gaussian fluctuations of the 2d {KPZ} equation.
\newblock {\em Stoch. Partial Differ. Equ. Anal. Comput.}, 8:150--185, 2020.

\bibitem{GQT}
Yu~Gu, Jeremy Quastel, and Li-Cheng Tsai.
\newblock Moments of the 2d {SHE} at criticality.
\newblock {\em Probab. Math. Phys.}, 2:179--219, 2021.

\bibitem{KRV}
Antti Kupiainen, R\'{e}mi Rhodes, and Vincent Vargas.
\newblock Integrability of {L}iouville theory: proof of the {DOZZ} formula.
\newblock {\em Ann. of Math. (2)}, 191(1):81--166, 2020.

\bibitem{LawlerIntersections}
Gregory~F. Lawler.
\newblock {\em Intersections of random walks}.
\newblock Modern Birkh{\"a}user Classics. Birkh{\"a}user/Springer, New York,
  2013.
\newblock Reprint of the 1996 edition.

\bibitem{LyZy21}
Dimitris Lygkonis and Nikos Zygouras.
\newblock Moments of the $2d$ directed polymer in the subcritical regime and a
  generalization of the {Erd\"{o}s-Taylor} theorem.
\newblock {\em arXiv:2109.06115}, 2021.

\bibitem{LyZy22}
Dimitris Lygkonis and Nikos Zygouras.
\newblock A multivariate extension of the {Erd\"{o}s-Taylor} theorem.
\newblock {\em arXiv:2202.08145}, 2022.

\bibitem{NaNa21}
Shuta Nakajima and Makoto Nakashima.
\newblock Fluctuation of two-dimesional stochastic heat equation and {KPZ}
  equation in subcritical regime for general initial conditions.
  \newblock {\em Electron. J. Probab.},
  paper 1, 38 pp.,  2023.

\bibitem{RV}
R\'{e}mi Rhodes and Vincent Vargas.
\newblock Gaussian multiplicative chaos and applications: a review.
\newblock {\em Probab. Surv.}, 11:315--392, 2014.

\bibitem{Sznitman}
Alain-Sol Sznitman.
\newblock {\em Brownian motion, obstacles and random media}.
\newblock Springer Monographs in Mathematics. Springer-Verlag, Berlin, 1998.

\end{thebibliography}
\appendix
\section{Proof of \eqref{eq:pnstar}}
\label{sec-pnstar}
First note that
$p_{2n}^\star\leq p_{2n}(0)$ since,
by the Cauchy-Schwarz inequality,
\[p_{2n}(x)=\sum_{y} p_n(x-y) p_n(y) \leq \sum_y p_n(y)^2 = p_{2n}(0).\]

Let $p^{(d)}_{2n}$ be the return probability of $d$-dimensional SRW to $0$. A direct computation gives
that $p^{(2)}_{2n}=(p^{(1)}_{2n})^2$
(see e.g. \cite[Page 184]{Durrett}).
We will show that $a_n=\sqrt{2n} p^{(1)}_{2n}$ is increasing. We have,
\[ a_n =\sqrt{2n} 2^{-2n}
\begin{pmatrix}
2n\\
n
\end{pmatrix}.\] Hence,
\begin{eqnarray*}
  \frac{a_{n+1}}{a_n}& = &\frac14
  \sqrt{\frac{n+1}{n}} \frac{ (2n+2)(2n+1)}{(n+1)^2}\\
  &=&
  \sqrt{1/(n(n+1))}  (n+ (n+1))/2.
\end{eqnarray*}
Since $(a+b)/2\geq \sqrt {ab}$, we conclude (using $a=n$ and $b=n+1$) that
$a_{n+1}/a_n\geq 1$.

Let $p_{2n+1}^{(1)}$ be the  probability of the $1$-dimensional SRW to come back to 1 in $2n+1$ steps. By the random walk representation
\cite[Remark in  Pg. 185]{Durrett},
we have that
$p_{2n+1}^\star \leq (p^{(1)}_{2n+1})^2$. A similar line of
argument to the above shows that $b_n=\sqrt{2n+1}p_{2n+1}^{(1)}$ is increasing in $n$. Indeed,
\begin{align*}
\frac{b_{n+1}}{b_n} & = \frac{1}{4} \frac{\sqrt{2n+3}}{\sqrt{2n+1}} \frac{(2n+3)(2n+2)}{(n+2)(n+1)}\\
& =  \frac{2n+3}{2\sqrt{(n+1)(n+2)}} \frac{\sqrt{(2n+3)(n+1)}}{\sqrt{(2n+1)(n+2)}},
\end{align*}
where the first fraction is bigger than 1 by the formula $(a+b)/2\geq \sqrt{ab}$, as well as the second fraction by expanding the products.

Now, we know from the local limit theorem  that $a_n$ and $b_n$ converge to $2/\sqrt{2\pi}$, thus they are always smaller than this limit.  This leads to \eqref{eq:pnstar}.
\qed
\section{Improved estimates on $U_N$}
\label{app-UN}
When $n$ is taken large enough, the estimate \eqref{eq:boundP2P2ndMomentC} can be improved as follows.
\begin{proposition} \label{prop:factor2ndMomentbis}
There exists $\varepsilon_{n}=\varepsilon(n,\hat \beta)\to 0$ such that as $n\to\infty$, uniformly for $N\geq n$,
\begin{equation} \label{eq:boundP2P2ndMoment}
 U_N(n) =  (1+\varepsilon_{n})\frac{\hat \beta^2}{\left(1-\hat \beta^2 \frac{\log n}{\log N}\right)^2} \frac{1}{ n \log N} .
\end{equation}
\end{proposition}
\begin{proof}
Since $(S_n^1-S_n^2)\eqlaw (S_{2n})$, we can write
\begin{equation} \label{eq:UNinSRW}
U_N(n+1) =
\sigma_N^2 \DE_{0}\left[ e^{ \b_N^2\sum_{i=1}^{n} \mathbf 1_{S_{2i} = 0}} \mathbf{1}_{S^1_{2n}= 0} \right].
\end{equation}
Consider $\ell =\ell_n = n^{1-\varepsilon_n}$ with $\varepsilon_n = \frac{1}{\log \log n}$, so that $\ell_n=o(n)$ and $\varepsilon_n\to 0$.

\textbf{First step:} As $n\to\infty$ with $n\leq N$,
\begin{equation} \label{eq:firstStep}
\DE_{0}\left[ e^{ \b_N^2\sum_{i=1}^{n} \mathbf 1_{S_{2i} = 0}} \mathbf{1}_{S_{2n}= 0} \right] \sim
\DE_{0}\left[ e^{ \b_N^2(\sum_{i=1}^{\ell} \mathbf 1_{S_{2i} = 0}+ \sum_{i=n-\ell}^{n} \mathbf 1_{S_{2i} = 0})}  \mathbf{1}_{S_{2n}= 0 } \right].
\end{equation}
We compute the norm of the difference which, using that $|e^{-x}-1|\leq |x|$ for $x\geq 0$, is less than
\begin{align*}
&\DE_{0}\left[ e^{ \b_N^2\sum_{i=1}^{n} \mathbf 1_{S_{2i} = 0}} \mathbf{1}_{S_{2n}= 0} \times \b_N^2 \sum_{j=\ell}^{n-\ell} \mathbf 1_{S_{2j} = 0} \right]\\
&= \b_N^2 \sum_{j=\ell}^{n-\ell} \DE_{0}\left[ e^{ \b_N^2\sum_{i=1}^{j} \mathbf 1_{S_{2i} = 0}}  \mathbf 1_{S_{2j} = 0}\right]   \DE_{0}\left[ e^{ \b_N^2\sum_{i=1}^{n-j} \mathbf 1_{S_{2i} = 0}}  \mathbf 1_{S_{2(n-j)} = 0} \right].
\end{align*}
where we have used Markov's property in the second line. By \eqref{eq:UNinSRW} and \eqref{eq:boundP2P2ndMomentC}, the last sum is smaller than
\begin{align*} C\b_N^2 \sum_{j=\ell}^{n-\ell} \frac{1}{j}   \frac{1}{n-j} \leq 2C\b_N^2 \sum_{j=\ell}^{n/2} \frac{1}{j}   \frac{1}{n/2} &\leq \frac{1}{n}C'\beta_N^2 \log\left(\frac{n}{\ell_n}\right)
\leq \frac{1}{n}C''\varepsilon_n =o(n^{-1}).
\end{align*}
Since the left hand side of \eqref{eq:firstStep} is bigger than $c n^{-1}$ for some constant $c>0$, this shows \eqref{eq:firstStep}.

\textbf{Second step:} As $n\to\infty$ with $n\leq N$,
\begin{equation} \label{eq:2ndStep}
\begin{aligned}
&\DE_{0}\left[ e^{ \b_N^2(\sum_{i=1}^{\ell} \mathbf 1_{S_{2i} = 0}+ \sum_{i=n-\ell}^{n} \mathbf 1_{S_{2i} = 0})}  \mathbf{1}_{S_{2n}= 0 } \right] \\
&\sim \DE_{0}\left[ e^{ \b_N^2\sum_{i=1}^{\ell} \mathbf{1}_{S_{2i}=0}} \right] \DE_{0}\left[ e^{ \b_N^2\sum_{i=n-\ell}^{n} \mathbf 1_{S_{2i} = 0}}\mathbf 1_{S_{2n} = 0} \right].
\end{aligned}
\end{equation}
By Markov's property, we can write the LHS of \eqref{eq:2ndStep} as
\begin{align*}
& \sum_{x\in \mathbb Z^2} \DE_{0}\left[ e^{ \b_N^2\sum_{i=1}^{\ell} \mathbf 1_{S_{2i} = 0}} \mathbf{1}_{S_{2\ell} = x} \DE_x \left[e^{\b_N^2  \sum_{i=n-2\ell}^{n-\ell} \mathbf 1_{S_{2i} = 0}}  \mathbf{1}_{S_{2{n-\ell}}= 0 }\right]\right]\\
&  = \sum_{x\in \mathbb Z^2} \DE_{0}\left[ e^{ \b_N^2\sum_{i=1}^{\ell} \mathbf 1_{S_{2i} = 0}} \mathbf{1}_{S_{2\ell} = x} \DE_0 \left[e^{\b_N^2  \sum_{i=1}^{\ell} \mathbf 1_{S_{2i} = 0}}   \mathbf{1}_{S_{2n-\ell}=x} \right]\right].
\end{align*}
Therefore the difference in \eqref{eq:2ndStep} writes $\sum_{x\in \mathbb Z^2} \Delta_x$ with
\begin{align*}
& \Delta_x := \DE_{0}\left[ e^{ \b_N^2\sum_{i=1}^{\ell} \mathbf 1_{S_{2i} = 0}} \mathbf{1}_{S_{2\ell} = x}\DE_0 \left[e^{\b_N^2  \sum_{i=1}^{\ell} \mathbf 1_{S_{2i} = 0}}  \left(\mathbf{1}_{S_{2{n-\ell}}= 0 } - \mathbf{1}_{S_{2{n-\ell}}= x }\right)\right]\right].
\end{align*}
Since $\DE_0 \left[e^{\b_N^2  \sum_{i=1}^{\ell} \mathbf 1_{S_{2i} = 0}} \right] \leq C(\hat \b)$ by  \eqref{eq:bound2ndMoment}, we have
\[
\sum_{|x| > \sqrt{\ell} n^{\varepsilon/4}} |\Delta_x| \leq  C\sum_{|x| > \sqrt{\ell} n^{\varepsilon/4}} \DE_{0}\left[ e^{ \b_N^2\sum_{i=1}^{\ell} \mathbf 1_{S_{2i} = 0}} \mathbf{1}_{S_{2\ell} = x}\right].
\]
By H\"older's inequality with $p^{-1}+q^{-1}=1$, and $p$ small enough so that $\sqrt p \hat{\beta} <1$,
\begin{align*}
\DE_{0}\left[ e^{ \b_N^2\sum_{i=1}^{\ell} \mathbf 1_{S_{2i} = 0}} \mathbf{1}_{S_{2\ell} = x}\right] &\leq \DE_{0}\left[ e^{ p\b_N^2\sum_{i=1}^{\ell} \mathbf 1_{S_{2i} = 0}}\right]^{\frac{1}{p}} p_{2\ell}(x)^{\frac 1 q}
\leq C(\hat \beta) \ell_n^{-1}e^{-\frac{1}{2q} \frac{|x|^2}{\ell_n}},
\end{align*}
for $n$ large enough. Therefore,
\begin{align*}
\sum_{|x| > \sqrt{\ell} n^{\varepsilon/4}} |\Delta_x| &\leq C \sum_{|x| > \sqrt{\ell} n^{\varepsilon/4}}  \ell_n^{-1}e^{-\frac{1}{2q} \frac{|x|^2}{\ell_n}},
 \leq C e^{-\frac{1}{2q} n^{\varepsilon/2}} = o(n^{-1}).
\end{align*}
We now estimate the sum on $\Delta_x$ for $|x| \leq \sqrt{\ell} n^{\varepsilon/4}$. We start by bounding the expectation inside the definition of $\Delta_x$:
\begin{equation} \label{eq:insideDelta}
\begin{aligned}
&\DE_0 \left[e^{\b_N^2  \sum_{i=1}^{\ell} \mathbf 1_{S_{2i} = 0}}  \left(\mathbf{1}_{S_{2{n-\ell}}= 0 } - \mathbf{1}_{S_{2{n-\ell}}= x }\right)\right]\\
&= \sum_{y\in \mathbb Z^2} \DE_0 \left[e^{\b_N^2  \sum_{i=1}^{\ell} \mathbf 1_{S_{2i} = 0}} \mathbf{1}_{S_{\ell}= y} \right]\left(p_{2n-2\ell}(y)-p_{2n-2\ell}(y-x)\right).
\end{aligned}
\end{equation}
By the same argument as above, we can prove that the above sum restricted to $|y|\geq \sqrt \ell n^{\varepsilon/4}$ is negligible with respect to $n^{-1}$, uniformly for $|x| \leq \sqrt{\ell} n^{\varepsilon/4}$. On the other hand, by the local limit theorem we have
\[
\sup_{|x|\leq \sqrt \ell n^{\varepsilon/4},|y|\leq \sqrt \ell n^{\varepsilon/4}} \left|p_{2n-2\ell}(y)-p_{2n-2\ell}(y-x)\right| = o(n^{-1}).
\]
since $\ell_n n^{\varepsilon/2} = n^{1-\varepsilon_n/2}=o(n)$. Thus, the quantity in \eqref{eq:insideDelta} is bounded  uniformly for $|x|\leq \sqrt \ell n^{\varepsilon/4}$ by
\[ \DE_0 \left[e^{\b_N^2  \sum_{i=1}^{\ell} \mathbf 1_{S_{2i} = 0}}  \right] \times o(n^{-1})=o(n^{-1}).\]
This completes the proof of \eqref{eq:2ndStep}.

\textbf{Third step:} As $n\to\infty$ with $n\leq N$,
\begin{equation} \label{eq:3rdstep}
\DE_{0}\left[ e^{ \b_N^2\sum_{i=n-\ell}^{n} \mathbf 1_{S_{2i} = 0}}\mathbf 1_{S_{2n} = 0} \right] \sim \DE_{0}\left[ e^{ \b_N^2\sum_{i=1}^{\ell} \mathbf 1_{S_{2i} = 0}} \right] p_{2n}(0).
\end{equation}
Equivalence \eqref{eq:3rdstep} can be proven by following the same line of arguments as used to prove \eqref{eq:2ndStep}, hence we omit its proof.

Now, combining the three steps leads to the equivalence
\[
\DE_{0}\left[ e^{ \b_N^2\sum_{i=1}^{n} \mathbf 1_{S_{2i} = 0}} \mathbf{1}_{S_{2n}= 0} \right]\sim \DE_{0}\left[ e^{ \b_N^2\sum_{i=1}^{\ell} \mathbf 1_{S_{2i} = 0}} \right]^2 p_{2n}(0).
\]
By \eqref{eq:asympt2ndMoment}, as $\log \ell \sim \log n$, we have
\[
 \DE_{0}\left[ e^{ \b_N^2\sum_{i=1}^{\ell} \mathbf 1_{S_{2i} = 0}} \right] \sim \frac{1}{1-\hat \beta^2 \frac{\log n}{\log N}},
\]
and so \eqref{eq:boundP2P2ndMoment} follows from \eqref{eq:UNinSRW} and the last two displays.
\end{proof}

\section{Khas'minskii's lemma for discrete Markov chains}
The following theorem is another discrete analogue of Khas'minskii's lemma, compare with Lemma \ref{lem:modKhas}.
\begin{theorem} \label{th:discrete_Khas}
Let $(Y_n)_n$ be any markov chain on a discrete state-space $E$ and let $f:E \rightarrow \mathbb R_+$. Then for all $k\in\mathbb N$,
if
\begin{equation} \label{eq:def_eta_appendix}
\eta_0:=\sup_{x\in E} \DE_{x} \left[\sum_{n=1}^k (e^{f(Y_n)}-1)\right] < 1,
\end{equation}
one has
\begin{equation}
\sup_{x\in E} \DE_{x} \left[e^{\sum_{n=1}^k f(Y_n)}\right] \leq \frac{1}{1-\eta_0}.
\end{equation}
\end{theorem}
\begin{proof} Denote by $D_n = e^{f(Y_n)}-1$. We have,
\begin{align*}
& \DE_{x} \left[e^{\sum_{n=1}^N f(Y_n)}\right] = \DE_{x} \left[\prod_{n=1}^N (1+D_n)\right]
 = \sum_{p=0}^\infty  \sum_{1\leq n_1 < \dots<  n_p \leq k}\DE_{x} \left[ \prod_{i=1}^p D_{n_i}\right]\\
& = \sum_{p=0}^\infty  \sum_{1\leq n_1 < \dots < n_{p-1} \leq k} \DE_{x}
\left[
\prod_{i=1}^{p-1} D_{n_i} \DE_{Y_{n_{p-1}}}
\left[\sum_{n=1}^{k-n_{p-1}} D_n \right]
\right]\\
& \stackrel{\eqref{eq:def_eta_appendix}}{\leq}\sum_{p=0}^\infty  \eta_0 \sum_{1\leq n_1 < \dots < n_{p-1} \leq k} \DE_{x}
\left[
\prod_{i=1}^{p-1} D_{n_i}
\right]
 \leq \dots \leq
\sum_{p=0}^\infty  \eta_0^p = \frac{1}{1-\eta_0}.
\end{align*}
\end{proof}
\begin{corollary} \label{cor:discrete_Khas}
Let $(Y_n)_n$ be any markov chain on a discrete state-space $E$ and let $f:E \rightarrow [0,1]$. Then for all $k\in\mathbb N$,
if
\begin{equation} \label{eq:def_eta_appendixcor}
\eta_1:=\sup_{x\in E} \DE_{x} \left[\sum_{n=1}^k f(Y_n)\right] < 1,
\end{equation}
one has
\begin{equation}
\sup_{x\in E} \DE_{x} \left[e^{\sum_{n=1}^k f(Y_n)}\right] \leq \frac{1}{1-\eta_1}.
\end{equation}
\end{corollary}
\begin{proof}
Simply observe that $e^{f(x)}-1\leq e^c f(x)$ and apply Theorem \ref{th:discrete_Khas}.
\end{proof}

\end{document}